\documentclass{amsart}

\usepackage{enumerate, amsmath, amsfonts, amssymb, amsthm, wasysym, graphics, graphicx, xcolor, url, hyperref, hypcap, a4wide, pdflscape, multido, overpic}
\hypersetup{colorlinks=true, citecolor=darkblue, linkcolor=darkblue}
\usepackage[all]{xy}
\usepackage{tikz}\usetikzlibrary{trees,snakes,shapes,arrows,matrix,calc}
\graphicspath{{figures/}}

\title[The diameter of type~$D$ associahedra and the non-leaving-face property]{The diameter of type~$D$ associahedra \\ and the non-leaving-face property}

\author[C.~Ceballos]{Cesar Ceballos$^{\star}$} 
\address[C.~Ceballos]{Department of Mathematics and Statistics, York University, Toronto}
\email{ceballos@mathstat.yorku.ca}
\urladdr{http://garsia.math.yorku.ca/~ceballos/}
\thanks{$^\star$CC was supported by the government of Canada through an NSERC Banting Postdoctoral Fellowship. He was also supported by a York University research grant.}

\author[V.~Pilaud]{Vincent Pilaud$^{\ddagger}$} 
\address[V.~Pilaud]{CNRS \& LIX, \'Ecole Polytechnique, Palaiseau}
\email{vincent.pilaud@lix.polytechnique.fr}
\urladdr{http://www.lix.polytechnique.fr/~pilaud/}
\thanks{$^\ddagger$VP was partially supported by the spanish MICINN grant MTM2011-22792 and the french ANR grant EGOS (12 JS02 002 01).}

\newtheorem{theorem}{Theorem}

\newtheorem{proposition}[theorem]{Proposition}
\newtheorem{lemma}[theorem]{Lemma}
\newtheorem{definition}[theorem]{Definition}

\theoremstyle{definition}

\newtheorem{remark}[theorem]{Remark}

\newcommand{\R}{\mathbb{R}} 
\newcommand{\configD}{\mathbb D} 
\newcommand{\disk}{D} 

\newcommand{\set}[2]{\left\{ #1 \;\middle|\; #2 \right\}} 
\newcommand{\ssm}{\smallsetminus} 
\newcommand{\eqdef}{\mbox{\,\raisebox{0.2ex}{\scriptsize\ensuremath{\mathrm:}}\ensuremath{=}\,}} 
\newcommand{\diag}[2]{[#1,#2]} 
\newcommand{\diagD}[2]{\ifthenelse{\equal{#2}{L}}{#1^{\textsc l}}{#1^{\textsc r}}} 
\newcommand{\starL}{\diagD{S}{L}} 
\newcommand{\starR}{\diagD{S}{R}} 
\newcommand{\pseudoquadrangle}{\raisebox{-1pt}{\includegraphics{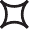}}} 
\newcommand{\Asso}{\mathsf{Asso}} 
\newcommand{\normalize}{\mathrm{N}} 

\DeclareMathOperator{\conv}{conv} 

\newcommand{\fref}[1]{Figure~\ref{#1}} 
\newcommand{\ie}{\textit{i.e.}~} 
\newcommand{\eg}{\textit{e.g.}~} 
\newcommand{\para}[1]{\medskip\noindent\textbf{#1~---}} 
\definecolor{darkblue}{rgb}{0,0,0.7} 
\newcommand{\darkblue}{\color{darkblue}} 
\newcommand{\defn}[1]{\emph{\darkblue #1}} 
\usepackage{todonotes}


\begin{document}

\begin{abstract}
Generalized associahedra were introduced by S. Fomin and A. Zelevinsky in connection to finite type cluster algebras. Following recent work of L.~Pournin in types~$A$ and~$B$, this paper focuses on geodesic properties of generalized associahedra. We prove that the graph diameter of the $n$-dimensional associahedron of type~$D$ is precisely~$2n-2$ for all~$n$ greater than~$1$. Furthermore, we show that all type~$BCD$ associahedra have the non-leaving-face property, that is, any geodesic connecting two vertices in the graph of the polytope stays in the minimal face containing both. This property was already proven by D.~Sleator, R.~Tarjan and W.~Thurston for associahedra of type $A$. In contrast, we present relevant examples related to the associahedron that do not always satisfy this property.

\smallskip
\noindent
{\sc keywords.}
Associahedron -- graph diameter -- pseudotriangulation -- flip graph.
\end{abstract}

\vspace*{-.8cm}

\maketitle

\vspace*{-.3cm}


\section{Introduction}

The associahedron is a convex polytope whose vertices are in correspondence with triangulations of a convex polygon and whose edges are flips among them. 
Motivated by efficiency of repeated access and information update in binary search trees, D.~Sleator, R.~Tarjan and W.~Thurston~\cite{SleatorTarjanThurston} showed that the diameter of the $n$-dimensional associahedron is at most~$2n-4$ for~$n$ greater than~$9$, and used arguments in hyperbolic geometry to prove that this bound is tight when~$n$ is large enough. They also conjectured that the diameter is~$2n-4$ for all~$n$ greater than~$9$. This conjecture was recently settled using purely combinatorial arguments by L.~Pournin~\cite{Pournin}, who explicitly exhibited two triangulations realizing this maximal distance. 

Associahedra are considered as one of the most important families of examples in polytope theory~\cite{LoeraRambauSantos,Ziegler1}. Besides their combinatorial beauty, they are of great importance in diverse areas in mathematics, computer science and physics~\cite{Stasheff, Stasheff-operads, TamariFestschrift}. One of the most significant appearances of associahedra is in the theory of cluster algebras initiated by S.~Fomin and A.~Zelevinsky~\cite{FominZelevinsky-ClusterAlgebrasI,FominZelevinsky-ClusterAlgebrasII}. They introduced a notion of generalized associahedra which extends the concept of associahedra to any Weyl group~\cite{FominZelevinsky-YSystems}. These essential objects encode the flip graphs of cluster algebras of finite type, and were realized as polytopes for the first time by F.~Chapoton, S.~Fomin and A.~Zelevinsky~\cite{ChapotonFominZelevinsky}. The concept of generalized associahedra was further extended to arbitrary Coxeter groups by N.~Reading in~\cite{Reading-CambrianLattices}. All these Coxeter associahedra, including those corresponding to Weyl groups, were realized as polytopes by C.~Hohlweg, C.~Lange and H.~Thomas~\cite{HohlwegLangeThomas} in connection with the Cambrian fans studied by N.~Reading and D.~Speyer in~\cite{ReadingSpeyer}.
Three interesting cases of generalized associahedra are the infinite families of type~$A$ (classical associahedra), of type~$B/C$ (cyclohedra), and of type~$D$.

In this paper, we show that the diameter of the $n$-dimensional associahedron of type~$D$ is precisely~$2n-2$ for all $n$ greater than~$1$. This is done using a convenient combinatorial model for type~$D$ associahedra in terms of centrally symmetric pseudotriangulations of a regular \mbox{$2n$-gon} with a small hole in the center. The same proof was presented independently by Y.~Lebrun in~\cite{Lebrun} using the combinatorial model in terms of decorated triangulations of the punctured $n$-gon arising from the work of S.~Fomin, M.~Shapiro and D.~Thurston~\cite{FominShapiroThurston}. As in~\cite{Pournin}, our methods are purely combinatorial and we explicitly describe two vertices of the polytope which are at maximal distance. In a recent preprint~\cite{Pournin-diameterTypeB}, L.~Pournin extends his method used in type~$A$ to derive the asymptotic diameter of type~$B$ associahedra: he shows that the diameter of the $n$-dimensional associahedron of type~$B$ is asymptotically~$5n/2$. As the type~$I_2(p)$ associahedron is a $(p+2)$-gon with diameter~${\lfloor p/2 \rfloor + 1}$, L.~Pournin's and our paper thus settle the problem of the graph diameters of generalized associahedra. For completeness, we gather in Table~\ref{table:asymptoticDiameters} the asymptotic diameters for type~$A$, $B/C$ and~$D$ associahedra, and in Table~\ref{table:smallDiameters} the precise diameters of the small rank generalized associahedra. The latter were reported in \cite{SleatorTarjanThurston} for type~$A$, in \cite{Pournin-diameterTypeB} for type~$B$, and were computed using C. Stump's Sage package on subword complexes for exceptional~types.

\begin{table}
    \centerline{
    \begin{tabular}{|l|c|c|c|}
    \hline
    \quad type & $A$ & $B/C$ & $D$ \\
    \quad asymptotic diameter$\quad$ & $\quad 2n-4 \quad $ & $\quad 5n/2 \quad$ & $\quad 2n-2 \quad$ \\
    \hline
    \end{tabular}
    }
    \bigskip
    \caption{Asymptotic diameters of type~$A$, $B/C$ and~$D$ associahedra.}
    \label{table:asymptoticDiameters}
\end{table}

\begin{table}
	\vspace{-.3cm}
    \centerline{
    \begin{tabular}{|p{1.4cm}|p{.1cm}p{.1cm}p{.1cm}p{.1cm}p{.1cm}p{.2cm}p{.2cm}p{.3cm}|p{.1cm}p{.1cm}p{.1cm}p{.1cm}p{.2cm}p{.2cm}p{.2cm}p{.2cm}p{.2cm}p{.2cm}p{.3cm}|c|p{.2cm}p{.2cm}p{.3cm}|p{.2cm}|p{.1cm}p{.3cm}|c|}
    \hline
    type & \multicolumn{8}{c|}{$A$} & \multicolumn{11}{c|}{$B/C$} & $D$ & \multicolumn{3}{c|}{$E$} & $F$ & \multicolumn{2}{c|}{$H$} & $I_2(p)$ \\
    rank & $2$ & $3$ & $4$ & $5$ & $6$ & $7$ & $8$ & $9$ & $2$ & $3$ & $4$ & $5$ & $6$ & $7$ & $8$ & $9$ & $10$ & $11$ & $12$ & $n$ & $6$ & $7$ & $8$ & $4$ & $3$ & $4$ & $2$ \\
    diameter & $2$ & $4$ & $5$ & $7$ & $9$ & $11$ & $12$ & $15$ & $3$ & $5$ & $7$ & $9$ & $11$ & $14$ & $16$ & $18$ & $21$ & $23$ & $25$ & $2n-2$ & $11$ & $14$ & $19$ & $8$ & $6$ & $10$ & $\lfloor p/2 \rfloor + 1$ \\
    \hline
    \end{tabular}
    }
    \bigskip
    \caption{Diameters of small rank associahedra of arbitrary finite types.}
    \label{table:smallDiameters}
\end{table}

In connection to this graph diameter question, we also show that all infinite families of associahedra of finite types have the non-leaving-face property, namely  every geodesic connecting two vertices in the graph of the polytope stays in the minimal face containing both. This is a known result of D.~Sleator, R.~Tarjan and W.~Thurston~\cite[Lemma~3]{SleatorTarjanThurston} for type $A$. Using similar \emph{normalization} ideas, we prove it type-by-type for types~$BCD$. By computer experiment with the computer software Sage~\cite{sage}, we also checked this property in the exceptional types~$E_6$,~$F_4$,~$H_3$ and~$H_4$. The remaining types~$E_7$ and~$E_8$ were still to be checked when N.~Williams announced a type-free proof of the non-leaving-face property for all generalized associahedra~\cite{Williams}. In contrast, we present five remarkable examples related to the associahedron which do not always satisfy this property: the pseudotriangulation polytopes, the multiassociahedra, the graph associahedra, the secondary polytopes, and the flip graphs on all triangulations of a point set.



\section{Pseudotriangulation model for type~$D$ associahedra}
\label{sec:Dmodel}

In this section, we present a combinatorial model for the type~$D_n$ associahedra~$\Asso(D_n)$ in terms of pseudotriangulations of a geometric configuration~$\configD_n$. The vertices of~$\Asso(D_n)$ correspond to centrally symmetric \emph{pseudotriangulations}, and its edges to \emph{flips} between them. In Remark~\ref{rem:typeD}, we compare our geometric interpretation to the classical models of type~$D$ cluster algebras presented by S.~Fomin and A.~Zelevinsky in~\cite[Section~3.5]{FominZelevinsky-YSystems}\cite[Section~12.4]{FominZelevinsky-ClusterAlgebrasII} and by S.~Fomin, M.~Shapiro and D.~Thurston in~\cite{FominShapiroThurston}.

\begin{figure}[b]
	\centerline{\includegraphics[scale=.8]{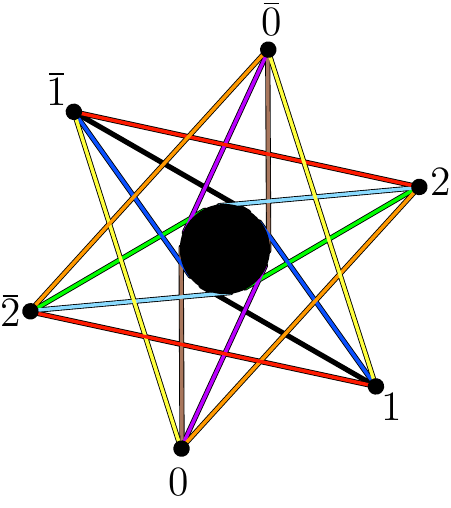} \quad \includegraphics[scale=.85]{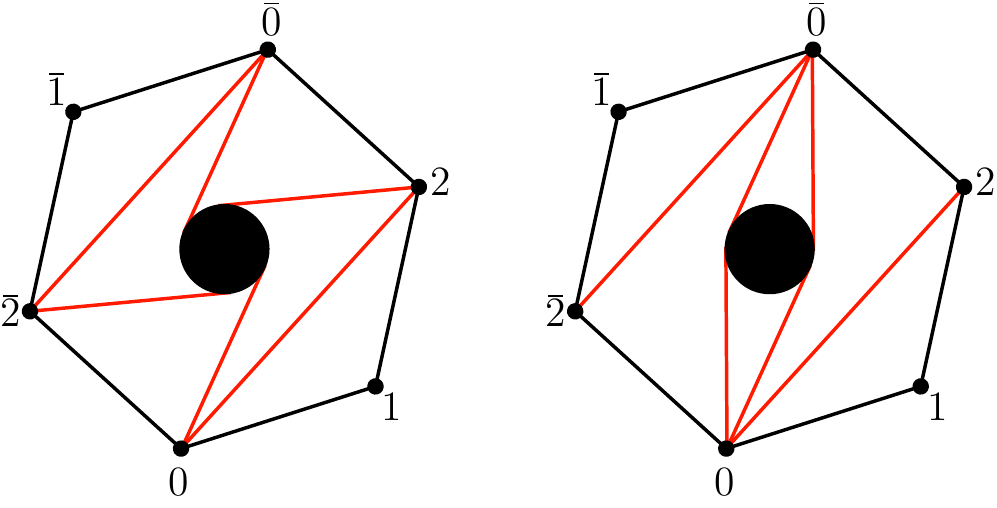}}
	\caption{The configuration~$\configD_3$ has $9$ centrally symmetric pairs of chords (left). A centrally symmetric pseudotriangulation~$T$ of~$\configD_3$ (middle). The centrally symmetric pseudotriangulation of~$\configD_3$ obtained from~$T$ by flipping the chords~$2^{\textsc r}$ and~$\bar 2^{\textsc r}$.}
	\label{fig:configurationDn&Flips}
\end{figure}

We consider a regular convex~$2n$-gon, together with a disk~$\disk$ placed at its center, whose radius is small enough such that~$\disk$ only intersects the long diagonals of the $2n$-gon. We denote by~$\configD_n$ the resulting configuration, see \fref{fig:configurationDn&Flips}. 
The vertices of~$\configD_n$ are labeled by $0,1,\dots, n-1, \overline 0, \overline 1,\dots ,\overline{n-1}$ in counterclockwise direction, such that two vertices $p$ and $\overline p$ are symmetric with respect to the center of the polygon. 
The \defn{chords} of~$\configD_n$ are all the diagonals of the $2n$-gon, except the long ones, plus all the segments tangent to the disk~$\disk$ and with one endpoint among the vertices of the \mbox{$2n$-gon}. Note that each vertex~$p$ is adjacent to two of the latter chords; we denote by~$\diagD{p}{L}$ (resp.~by~$\diagD{p}{R}$) the chord emanating from~$p$ which goes tangent on the left (resp.~right) to the disk~$\disk$, and we call these chords \defn{central}. 
For example, the four central chords that appear in \fref{fig:configurationDn&Flips}\,(middle) are $\diagD{0}{R},\diagD{\bar 0}{R},\diagD{2}{R}$ and $\diagD{\bar 2}{R}$. 
The faces of the type~$D_n$ associahedron can be interpreted geometrically on the configuration~$\configD_n$ as follows:
\begin{enumerate}[(i)]
\item Facets correspond to \defn{centrally symmetric pairs of (internal) chords} of the geometric configuration~$\configD_n$, see \fref{fig:configurationDn&Flips}\,(left).

\item Faces correspond to crossing-free centrally symmetric sets of chords. The face lattice corresponds to the reverse inclusion lattice on crossing-free centrally symmetric sets of chords.

\item Vertices correspond to \defn{centrally symmetric pseudotriangulations} of~$\configD_n$ (\ie inclusion maximal centrally symmetric crossing-free sets of chords of~$\configD_n$). Each pseudotriangulation of~$\configD_n$ contains exactly~$2n$ chords, and partitions $\conv(\configD_n) \ssm \disk$ into \defn{pseudotriangles}  (\ie interiors of simple closed curves with three convex corners related by three concave chains). See \fref{fig:configurationDn&Flips}\,(middle) and (right). We refer to~\cite{RoteSantosStreinu-survey} for a complete survey on pseudotriangulations, including their history, motivations, and applications.

\item \label{item:pseudoflip} Edges correspond to \defn{flips} of centrally symmetric pairs of chords between centrally symmetric pseudotriangulations of~$\configD_n$. A flip in a pseudotriangulation~$T$ replaces an internal chord~$e$ by the unique other internal chord~$f$ such that~$(T \ssm e) \cup f$ is again a pseudotriangulation of~$T$. Indeed, deleting~$e$ in~$T$ merges the two pseudotriangles of~$T$ incident to~$e$ into a pseudoquadrangle~$\pseudoquadrangle$ (\ie the interior of a simple closed curve with four convex corners related by four concave chains), and adding~$f$ splits the pseudoquadrangle~$\pseudoquadrangle$ into two new pseudotriangles. The chords~$e$ and~$f$ are the two unique chords which lie both in the interior of~$\pseudoquadrangle$ and on a geodesic between two opposite corners of~$\pseudoquadrangle$. We refer again to~\cite{RoteSantosStreinu-survey} for more details. 

For example, the two pseudotriangulations of \fref{fig:configurationDn&Flips} are related by a centrally symmetric pair of flips. We have represented different kinds of flips between centrally symmetric pseudotriangulations of the configuration~$\configD_n$ in \fref{fig:typeDflip}. Finally, \fref{fig:typeD3associahedron} shows the flip graph on centrally symmetric pseudotriangulations of~$\configD_3$, and Figures~\ref{fig:typeD4associahedron} and~\ref{fig:42pseudotriangulationsHorizontal} the one~for~$\configD_4$.

\enlargethispage{.3cm}
\begin{figure}[b]
	\centerline{\includegraphics[width=.92\textwidth]{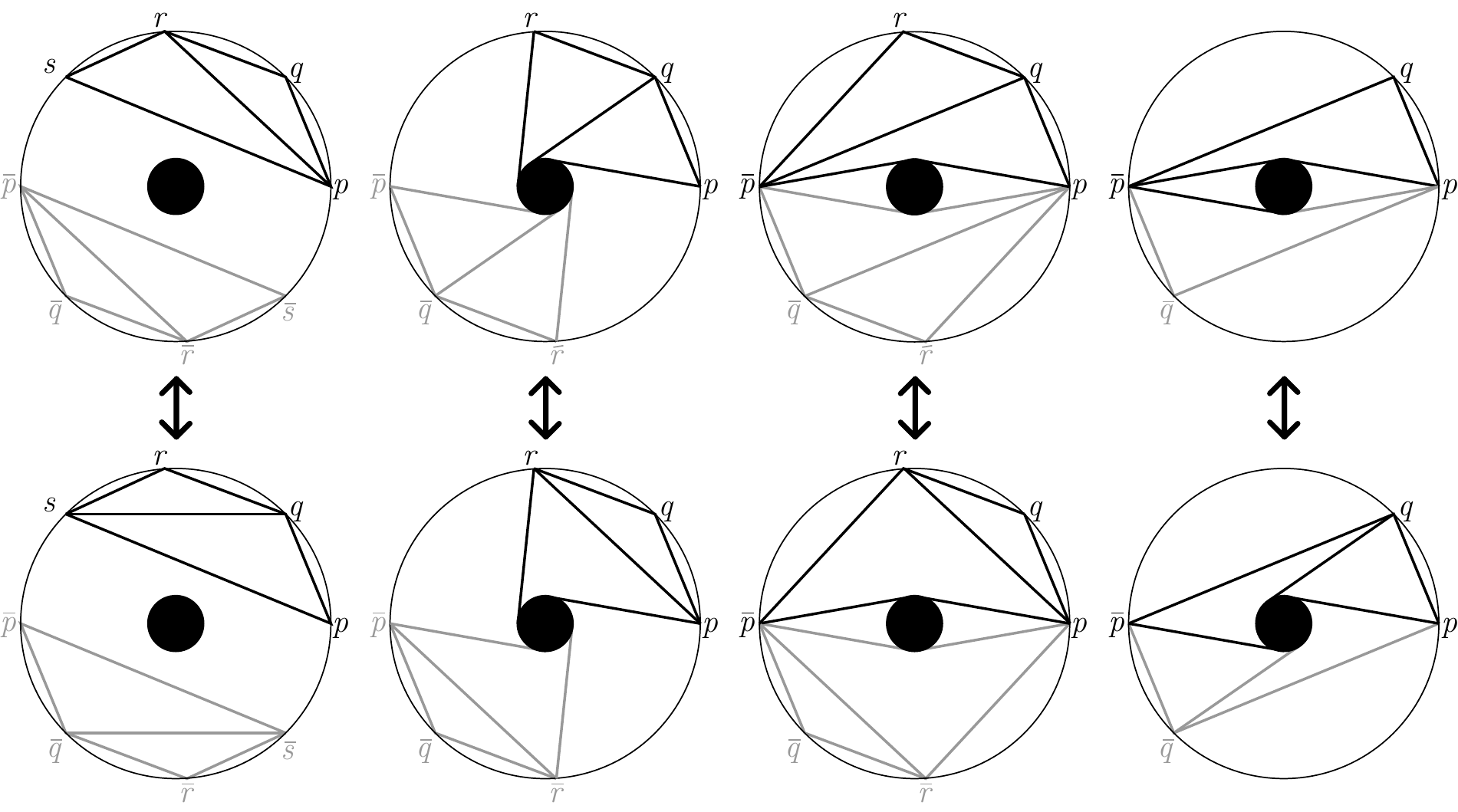}}
	\caption{Different kinds of flips in type~$D$.}
	\label{fig:typeDflip}
\end{figure}

\begin{figure}
	\centerline{\includegraphics[width=.6\textwidth]{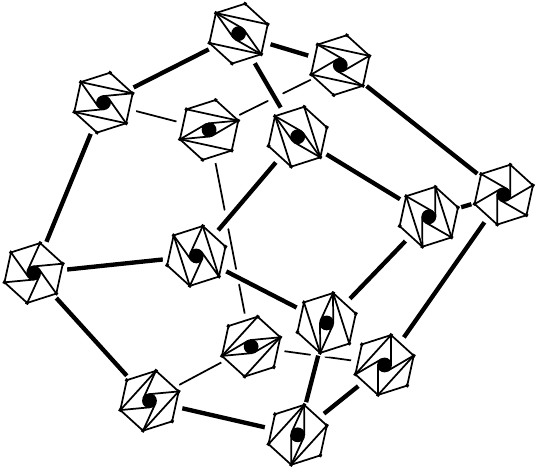}}
	\caption{The type~$D_3$ flip graph interpreted geometrically with centrally symmetric pseudotriangulations of~$\configD_3$. Note that this graph is the $1$-skeleton of the $3$-dimensional associahedron since~$D_3 = A_3$.}
	\label{fig:typeD3associahedron}
\end{figure}
\end{enumerate}

\begin{remark}
\label{rem:typeD}
Our geometric interpretation of type~$D$ associahedra slightly differs from that of S.~Fomin and A.~Zele\-vinsky in~\cite[Section~3.5]{FominZelevinsky-YSystems}\cite[Section~12.4]{FominZelevinsky-ClusterAlgebrasII}. Namely, to obtain their interpretation, we can just remove the disk in the configuration~$\configD_n$ and replace the centrally symmetric pairs of chords~$\{\diagD{p}{L}, \diagD{\bar p}{L}\}$ and~$\{\diagD{p}{R}, \diagD{\bar p}{R}\}$ by long diagonals~$\diag{p}{\bar p}$ colored in red and blue respectively. 
Another classical model for the type~$D$ cluster algebra is given by the decorated triangulations of the punctured $n$-gon from the work of S.~Fomin, M.~Shapiro and D.~Thurston~\cite{FominShapiroThurston}. We can obtain this model by folding our $2n$-gon around its central symmetry. In~\cite{Lebrun}, Y.~Lebrun gives the same proof of the diameter of the type~$D$ associahedron presented in the next section using this punctured $n$-gon model.
\end{remark}


\section{Diameter}

In this section, we present an explicit formula for the diameter of type~$D$ associahedra.

\begin{theorem}
\label{theo:diameter}
The diameter of the $n$-dimensional associahedron of type~$D$ is exactly~$2n-2$~for~all~$n$ greater than~$1$.
\end{theorem}

Define the \defn{left star}~$\starL$ to be the pseudotriangulation formed by all left central chords~$\diagD{p}{L}$ for~${p \in [n] \cup [\bar n]}$. Similarly, the \defn{right star}~$\starR$ is formed by all right central chords~$\diagD{p}{R}$ for~$p \in [n] \cup [\bar n]$.

\begin{lemma}
\label{lem:distanceStars}
The flip distance between the left star~$\starL$ and the right star~$\starR$ is precisely~$2n-2$.
\end{lemma}

\begin{proof}
We first claim that there is a path of exactly $2n-2$ flips between~$\starL$ and~$\starR$. This path is illustrated in \fref{fig:stars}, where the label on each arrow records the number of flips used to pass from one pseudotriangulation to the next one. Consider an arbitrary pair of opposite points~$\{p, \bar p\}$ of~$\configD_n$. Starting from~$\starL$, we successively flip the chord~$\diagD{(p+i)}{L}$ to the chord~$\diag{p}{p+i+1}$ for~$i \in [n-2]$, then the chord~$\diagD{p}{L}$ to the chord~$\diagD{(p-1)}{R}$, then the chord~$\diagD{(p-1)}{L}$ to the chord~$\diagD{\bar p}{R}$, and finally we successively flip the chord~$\diag{p}{\bar p - i}$ to the chord~$\diagD{(\bar p - i - 1)}{R}$ for~$i \in [n-2]$. Of course, we perform simultaneously the centrally symmetric flips of those just described. We thus used $(n-2)+1+1+(n-2) = 2n-2$ flips to transform~$\starL$ into~$\starR$.

\begin{figure}[h]
	\centerline{\includegraphics[width=1\textwidth]{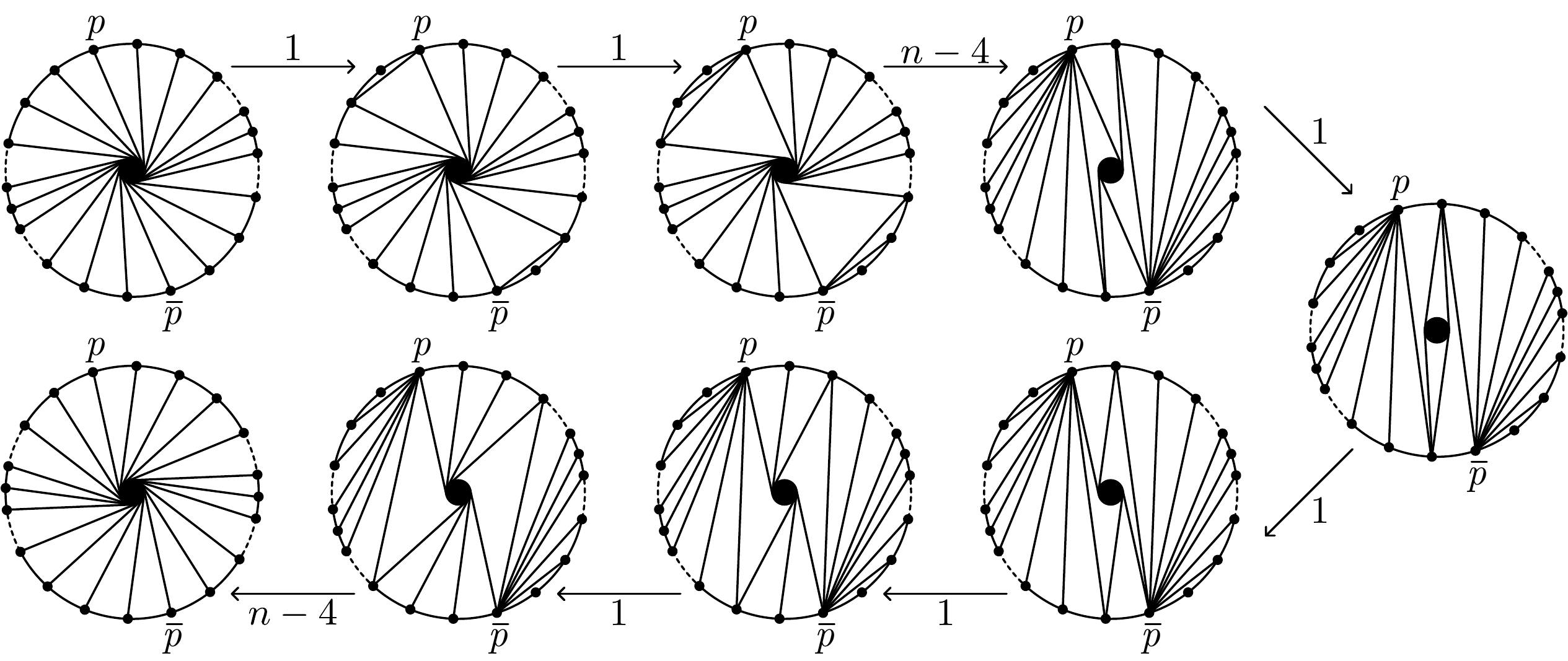}}
	\caption{The distance between the two stars~$\protect \starL$ and~$\protect \starR$ is $2n-2$.}
	\label{fig:stars}
\end{figure}

We now show that it is impossible to use less flips to transform~$\starL$ into~$\starR$. Observe that a pair of left central chords~$\{\diagD{p}{L}, \diagD{\bar p}{L}\}$ crosses any pair of right central chords~$\{\diagD{q}{R}, \diagD{\bar q}{R}\}$ except when~$p = q$. Therefore, if a centrally symmetric pseudotriangulation has strictly more than two pairs of left central chords then none of the possible flips produces a pair of right central chords. This means that we need to apply at least $n-2$ flips to the left star until we are able to make a flip that produces a pair of right central chords. Since the right star has $n$ pairs of right central chords, we need at least $n$ additional flips to produce it. This proves that any path between~$\starL$ and~$\starR$ uses at least $(n-2)+n= 2n-2$ flips.
\end{proof}

\begin{lemma}
\label{lem:distanceToStars}
Let~$T$ be a centrally symmetric pseudotriangulation of~$\configD_n$.
\begin{enumerate}[(i)]
\item If~$T$ contains~$\ell \ge 1$ centrally symmetric pairs of left central chords~$\{\diagD{p}{L}, \diagD{\bar p}{L}\}$, then~$T$ is precisely at distance~$n-\ell$ from the left star~$\starL$ and precisely at distance~$n+\ell-2$ from the right star~$\starR$.
\item If~$T$ contains~$r \ge 1$ centrally symmetric pairs of right central chords~$\{\diagD{p}{R}, \diagD{\bar p}{R}\}$, \mbox{then~$T$ is precisely} at distance~$n-r$ from the right star~$\starR$ and precisely at distance~$n+r-2$ from the left star~$\starL$.
\end{enumerate}
\end{lemma}

\begin{proof}
We prove Point (i) of the lemma, the other point follows by symmetry. Assume thus that~$T$ contains~$\ell \ge 1$ centrally symmetric pairs of left central chords.

Every non-left-central pair of chords in~$T$ needs to be flipped at least once in a path connecting~$T$ and~$\starL$. Moreover, it is possible to flip them one at a time such that each flip produces a pair of left central chords: at each step, flip a pair of centrally symmetric internal diagonals incident to a pseudotriangle touching the central disk. This shows that $T$ is precisely at distance~$n-\ell$ from the left star~$\starL$.

The proof that~$T$ is at distance~$n+\ell-2$ from~$\starR$ uses similar arguments to those in the proof of Lemma~\ref{lem:distanceStars}. If $\ell=1$, then~$T$ has exactly one pair of right central chords. In this case, the distance between~$T$ and~$\starR$ is~$n+\ell-2=n-1$ as desired. A geodesic between~$T$ and~$\starR$ can be obtained by successively flipping the $n-1$ non-right-central pairs of chords in $T$ such that each flip produces a pair of right central chords. If $\ell \geq 2$, then we need to apply at least~$\ell-2$ flips to~$T$ until we are able to make a flip that produces a pair of right central chords. Since the right star has~$n$ pairs of right central chords, we need at least~$n$ additional flips to produce it. Hence, the distance between~$T$ and~$\starR$ is~$n+\ell-2$.
\end{proof}

\begin{proof}[Proof of Theorem~\ref{theo:diameter}]
By Lemma~\ref{lem:distanceStars}, the diameter of the $n$-dimensional associahedron of type~$D$ is at least~$2n-2$. It remains to show that the distance between any two centrally symmetric pseudotriangulations~$T$ and~$\widetilde T$ is at most~$2n-2$. Let~$\ell$ and~$\widetilde \ell$ (resp.~$r$ and~$\widetilde r$) be the number of left (resp.~right) central pairs of chords in~$T$ and~$\widetilde T$ respectively. If~$T$ and~$\widetilde T$ contain a central pair of chords of the same kind, say left, then they can be connected by a path passing through the left star~$\starL$ of length 
\[n-\ell+n-\widetilde \ell \leq 2n-2.\]
If not, we can assume without loss of generality that $\ell\geq 1$ and $\widetilde r\geq 1$. By Lemma~\ref{lem:distanceToStars}, there are two paths connecting $T$ and $\widetilde T$ passing through the star triangulations~$\starL$ and~$\starR$ respectively~of~length 
\[
(n-\ell) + (n+\widetilde r-2)=2n-2-\ell+\widetilde r \qquad \text{and}\qquad (n+\ell-2)+(n-\widetilde r)=2n-2+\ell-\widetilde r.
\]
Clearly, one of these two numbers is less than or equal to $2n-2$. 
\end{proof}


\section{Non-leaving-face property}

\enlargethispage{.4cm}
This section is devoted to the following natural property related to diameter and graph distance on polytopes.

\begin{definition}
A polytope~$P$ has the \defn{non-leaving-face property} if any geodesic connecting two vertices in the graph of~$P$ stays in the minimal face of~$P$ containing both.
\end{definition}

Many classical polytopes have the non-leaving-face property:
\begin{enumerate}[(i)]
\item the $n$-gon: proper faces are segments;
\item the simplex: any two vertices are at distance~$1$;
\item any simplicial polytope: any two vertices belonging to a proper face are at distance~$1$;
\item the $n$-cube: the distance between any two vertices is the Hamming distance between them (\ie the number of coordinates where they differ);
\item the permutahedron: the distance between two permutations~$\sigma, \widetilde \sigma$ in the permutahedron is the number of inversions of~$\sigma^{-1}\widetilde \sigma$, and the minimal face containing~$\sigma$ and~$\widetilde \sigma$ is the ordered partition~$\set{\sigma([n_{i+1}] \ssm [n_i])}{i \in [p]}$, where~$0 = n_0 < n_1 < \dots < n_p = n$ is the finest subdivision of~$[n]$ such that~$\sigma([n_{i+1}] \ssm [n_i]) = \widetilde \sigma([n_{i+1}] \ssm [n_i])$ for all~$i \in [p]$; leaving the minimal face containing~$\sigma$ and~$\widetilde \sigma$ thus introduces useless inversions and therefore lengthen the way from~$\sigma$ to~$\widetilde \sigma$;
\item the associahedron: see~\cite[Lemma~3]{SleatorTarjanThurston} and Theorem~\ref{thm:nonLeavingFaceProperty};
\item the cyclohedron (which we refer to as type~$B/C$ associahedron): see Theorem~\ref{thm:nonLeavingFaceProperty}.
\end{enumerate}
However, there are also many examples of polytopes that do not satisfy the non-leaving-face property, for example a pyramid over a hexagon. Other examples related to associahedra are presented in Section~\ref{sec:ctrexms}.

Note that the non-leaving-face property implies the classical non-revisiting property (between any two vertices, there exists a path which does not leave and revisit any facet), which in turn implies the Hirsch bound on the diameter of the polytope~\cite{Klee}, see also~\cite{Santos-surveyHirsch}. The reverse implications are wrong, consider \eg a pyramid over a square. The Hirsch bound for generalized associahedra also follows from their vertex-decomposability~\cite{CeballosLabbeStump}, or from general results on flag polytopes by K.~Adiprasito and B.~Benedetti~\cite{AdiprasitoBenedetti}. We focus here on the non-leaving-face property and refer to~\cite{Santos-surveyHirsch} for further details.

In this section, we show that all associahedra of type~$A$, $B/C$, or $D$ have the non-leaving-face property. This property was proven in type~$A$ by D.~Sleator, R.~Tarjan, and W.~Thurston~\cite[Lemma~3]{SleatorTarjanThurston}. We generically use the term ``triangulation'' to refer to the geometric model for clusters in type~$A$, $B/C$ or~$D$: classical triangulations of the $(n+3)$-gon in type~$A_n$, centrally symmetric triangulations of the $2n$-gon in type~$B_n/C_n$, and centrally symmetric pseudotriangulations of the configuration~$\configD_n$ in type~$D_n$. Similarly, ``diagonal'' refers to a diagonal in type~$A_n$, to a centrally symmetric pair of diagonals or a long diagonal in type~$B_n/C_n$, and to a centrally symmetric pair of chords of the configuration~$\configD_n$ in type~$D_n$.

\begin{theorem}
\label{thm:nonLeavingFaceProperty}
All associahedra of type~$A$, $B/C$, $D$ have the non-leaving-face property. In other words, no common diagonal between two triangulations~$T, \widetilde T$  is flipped in a geodesic between~$T$~and~$\widetilde T$.
\end{theorem}

\begin{remark}
Since the type~$I_2(n)$ associahedra are all polygons, this proposition proves that all infinite families of associahedra have the non-leaving-face property. 
We also checked this property for the exceptional types~$E_6$, $F_4$, $H_3$ and~$H_4$. We use a variant of breadth first search algorithm to compute all geodesics from a given point and we check along the way that these geodesics stay in the smallest face containing their endpoints. For concrete computations, we use the subword complex model for finite type cluster algebras~\cite{CeballosLabbeStump} as implemented by C.~Stump in Sage~\cite{sage}. This time-consuming verification (about half a day for~$E_6$) was still to be done on types~$E_7$ and~$E_8$ when N.~Williams announced a type-free proof of the non-leaving-face property for all generalized associahedra~\cite{Williams} using sortable elements and Cambrian lattices~\cite{Reading-CambrianLattices}. Below, we stick to our original type-by-type proof which involves more geometric arguments.
\end{remark}

Theorem~\ref{thm:nonLeavingFaceProperty} is a consequence of the following stronger statement. In fact Proposition~\ref{prop:enteringFaceProperty} even shows that any path leaving the minimal face containing two vertices has at least~$2$ more steps than the geodesic connecting them.

\begin{proposition}
\label{prop:enteringFaceProperty}
Let~$T$ and~$\widetilde T$ be two triangulations of type~$A$, $B/C$, or $D$, and~$\chi$ be a diagonal in~$T \ssm \widetilde T$. If the flip of~$\chi$ in~$T$ produces a diagonal that belongs to~$\widetilde T$, then there exists a geodesic between~$T$ and~$\widetilde T$ which starts by the flip of~$\chi$.
\end{proposition}

The proof of Proposition~\ref{prop:enteringFaceProperty} relies on a normalization argument, generalizing the normalization of~\cite[Lemma~3]{SleatorTarjanThurston} for classical triangulations. We first introduce the normalization on triangulations of types~$A$, $B/C$ and~$D$, and then return to the proof of Theorem~\ref{thm:nonLeavingFaceProperty} and Proposition~\ref{prop:enteringFaceProperty} in Section~\ref{subsec:proofs}.


\subsection{Normalization}

A normalization is a useful tool to transform a geodesic between two triangulations into a geodesic starting with a prescribed flip. In particular, it is a projection from the graph of the associahedron to the graph of one of its facets. Recall that we use the terms ``triangulation'' and ``diagonal'' generically for the corresponding geometric models in types~$A$, $B/C$, and~$D$. 

\begin{proposition}
\label{prop:normalization}
For any type~$A$, $B/C$ or~$D$, and for any diagonal~$\chi$, there exists a \defn{normalization}~$\normalize_\chi$, that is, a map~$T \mapsto \normalize_\chi(T)$ satisfying the following properties:
\begin{enumerate}[(P1)]
\setcounter{enumi}{-1}
\item for any triangulation~$T$, the normalization~$\normalize_\chi(T)$ is a triangulation containing~$\chi$;
\item if~$\chi \in T$, then~$\normalize_\chi(T) = T$;
\item if~$T, T'$ are two adjacent triangulations, then~$\normalize_\chi(T)$ and~$\normalize_\chi(T')$ coincide or are adjacent;
\item if~$T, T'$ are two adjacent triangulations with~$\chi \in T' \ssm T$, then~${\normalize_\chi(T) = \normalize_\chi(T') = T'}$.
\end{enumerate}
\end{proposition}

We prove this proposition type-by-type using the geometric models in types~$A$, $B/C$, and~$D$. We refer to the recent preprint of N.~Williams~\cite{Williams} for a type-free proof.

\para{Type~$A$}
We use the same definition as~\cite[Lemma~3]{SleatorTarjanThurston}. The normalization~$\normalize_\chi$ with respect to a diagonal~$\chi \eqdef \diag{p}{q}$ transforms a triangulation~$T$ to the triangulation~$\normalize_\chi(T)$ obtained by deleting all diagonals of~$T$ crossing~$\chi$ and filling in the resulting empty region~$R_\chi(T)$ by the triangulation where all diagonals are incident to~$p$. \fref{fig:typeANormalization} illustrates this normalization.

\begin{figure}[h]
	\centerline{\includegraphics[scale=.61]{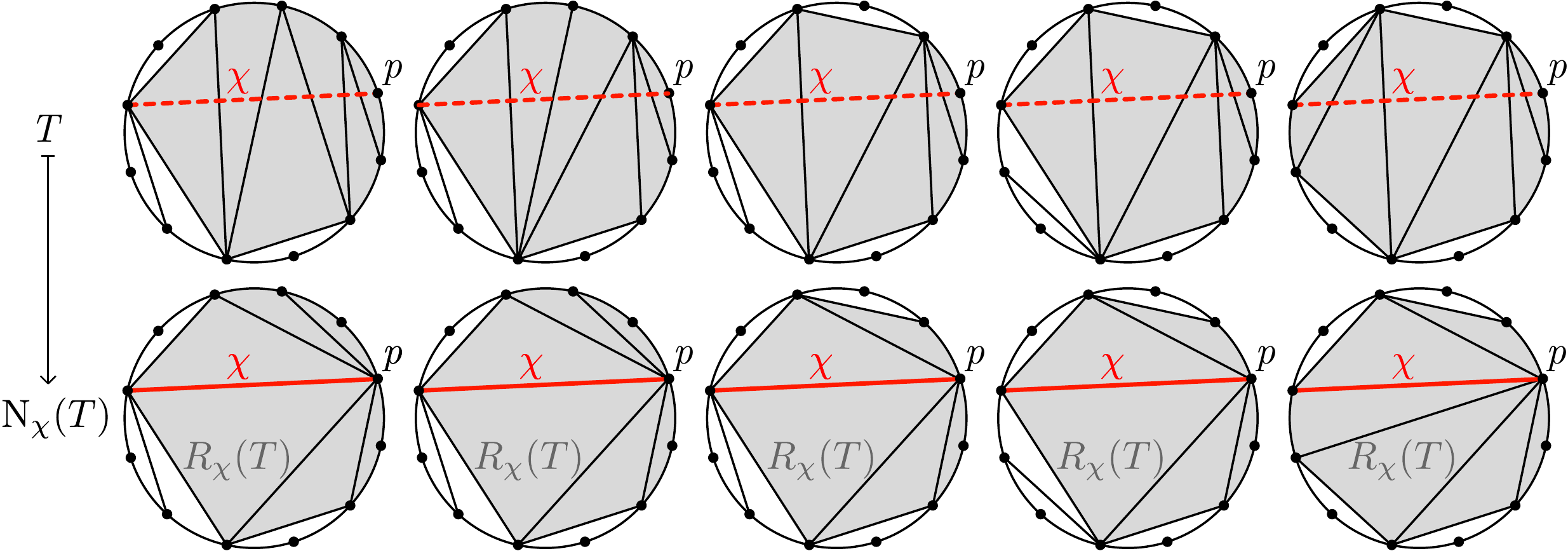}}
	\caption{Normalization map for type~$A$.}
	\label{fig:typeANormalization}
\end{figure}

Properties~(P0) and~(P1) are clear by construction. For Properties~(P2) and~(P3), consider two adjacent triangulations~$T$ and~$T'$, let~$\delta \in T$ and~$\delta' \in T'$ be such that~$T \ssm \{\delta\} = T' \ssm \{\delta'\}$, and let~$Q$ denote the quadrangle with diagonals~$\delta$ and~$\delta'$. If $\chi$ crosses both~$\delta$ and~$\delta'$, then the quadrangle~$Q$ lies in the retriangulated region~$R_\chi(T) = R_\chi(T')$, so that~$\normalize_\chi(T) = \normalize_\chi(T')$. If~$\chi$ crosses neither~$\delta$ nor~$\delta'$, then the quadrangle~$Q$ is disjoint from the retriangulated region~${R_\chi(T) = R_\chi(T')}$ and~${\normalize_\chi(T) \ssm \{\delta\} = \normalize_\chi(T') \ssm \{\delta'\}}$. If~$\chi$ crosses~$\delta$ but not~$\delta'$, and~$\chi \ne \delta'$, then~$\normalize_\chi(T) = \normalize_\chi(T')$ if~$p$ is an endpoint of~$\delta'$, and~$\normalize_\chi(T) \ssm \{\diag{p}{x}\} = \normalize_\chi(T') \ssm \{\delta'\}$ otherwise, where~$x$ denotes the endpoint of~$\delta$ separated from~$\chi$ by~$\delta'$. Finally, if~$\chi = \delta'$, then~$R_\chi(T) = Q$ while~$R_\chi(T') = \varnothing$, so that~$\normalize_\chi(T) = \normalize_\chi(T') = T'$.

To extend~$\normalize_\chi$ to types~$B/C$ and~$D$, it is convenient to consider the following equivalent description of this normalization: imagine that all diagonals of~$T$ are rubber bands attached to their endpoints and pull the rubber bands crossed by $\chi$ along~$\chi$ towards~$p$ to obtain~$\normalize_\chi(T)$. In other words, each diagonal~$\diag{x}{y}$ of~$T$ crossing~$\chi$ is replaced by~$\diag{x}{p}$ and~$\diag{p}{y}$.

\para{Type~$B/C$}
For type~$B/C$, we distinguish whether or not~$\chi$ is a long diagonal to define the normalization~$\normalize_\chi$:
\begin{itemize}
\item The normalization~$\normalize_\chi$ with respect to a centrally symmetric pair~$\chi \eqdef \{\diag{p}{q}, \diag{\bar p}{\bar q}\}$ of distinct diagonals transforms a centrally symmetric triangulation~$T$ into the centrally symmetric triangulation~$\normalize_\chi(T)$ obtained by deleting all diagonals of~$T$ crossing~$\chi$ and filling the resulting empty region~$R_\chi(T)$ by a centrally symmetric triangulation where all diagonals are incident to the points~$p$ and~$\bar p$ (note that the region~$R_\chi(T)$ need not be connected). To define the triangulation replacing~$T$ inside~$R_\chi(T)$, imagine that all diagonals of~$T$ are rubber bands attached to their endpoints and pull the rubber bands crossed by $\chi$ along~$\chi$ towards~$p$ and~$\bar p$ to obtain~$\normalize_\chi(T)$. In other words, each diagonal~$\diag{x}{y}$ of~$T$ crossing~$\chi$ is replaced by:
\begin{itemize}
\item $\diag{x}{p}$ and~$\diag{p}{y}$ if it only crosses~$\diag{p}{q}$ (and similarly with~$\diag{\bar p}{\bar q}$);
\item $\diag{x}{p}$, $\diag{p}{\bar p}$, and~$\diag{p}{y}$ if it crosses both~$\diag{p}{q}$ and~$\diag{\bar p}{\bar q}$, and~$\diag{p}{q}$ separates~$x$ from the origin while~$\diag{\bar p}{\bar q}$ separates~$y$ from the origin.
\end{itemize}
\item The normalization with respect to a long diagonal~$\chi \eqdef \diag{p}{\bar p}$ of the $2n$-gon transforms a centrally symmetric triangulation~$T$ to the centrally symmetric triangulation~$\normalize_\chi(T)$ obtained by deleting all diagonals of~$T$ crossing~$\chi$ and filling in the resulting empty region~$R_\chi(T)$ by the triangulation where all vertices clockwise between $p$ and~$\bar p$ are connected to~$\bar p$, while all vertices clockwise between $\bar p$ and~$p$ are connected to~$p$. 
\end{itemize}
\fref{fig:typeBNormalization} illustrates this normalization, where the three leftmost pictures correspond to the first situation, while the two rightmost correspond to the second situation. As in type~$A$, a straightforward case analysis shows that~$\normalize_\chi$ indeed defines a normalization in both situations.

\begin{figure}[h]
	\centerline{\includegraphics[scale=.61]{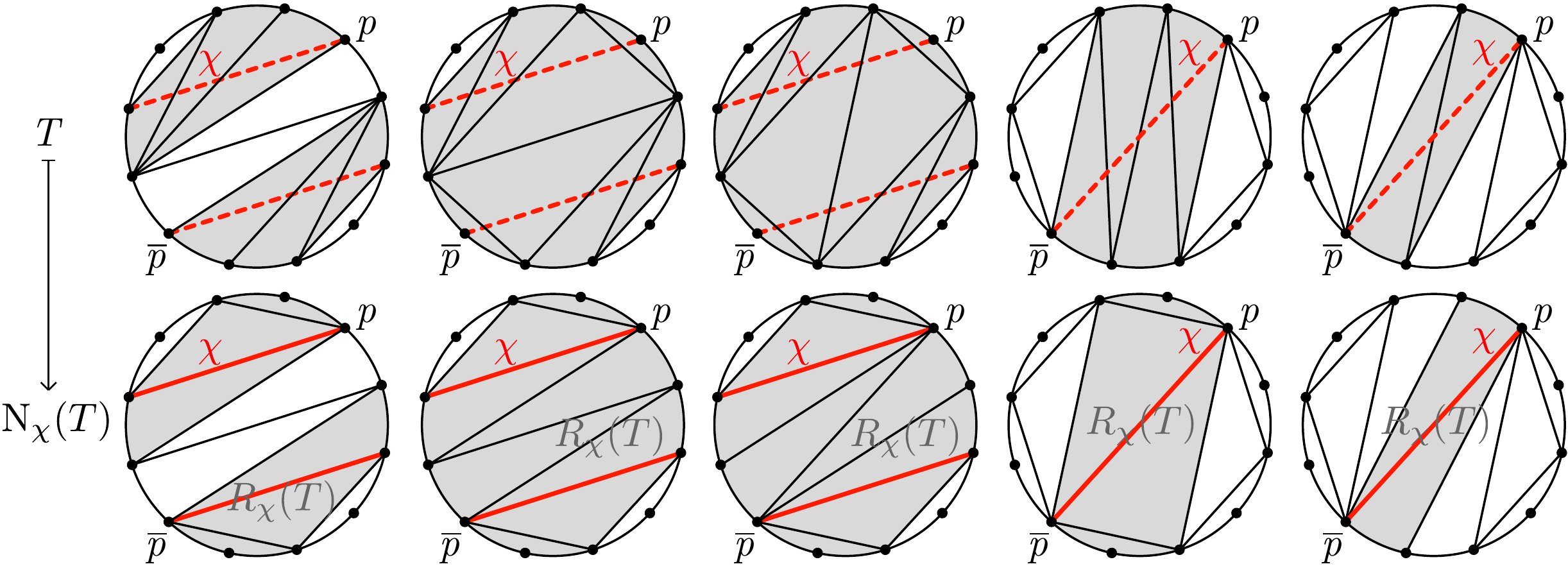}}
	\caption{Normalization map for type~$B/C$.}
	\label{fig:typeBNormalization}
\end{figure}

\para{Type~$D$}
For type~$D$, we distinguish whether or not~$\chi$ is a pair of central chords to define the normalization~$\normalize_\chi$:
\begin{itemize}
\item The normalization~$\normalize_\chi$ with respect to a centrally symmetric pair~$\chi \eqdef \{\diag{p}{q}, \diag{\bar p}{\bar q}\}$ of diagonals of the $2n$-gon transforms a centrally symmetric pseudotriangulation~$T$ into the centrally symmetric pseudotriangulation~$\normalize_\chi(T)$ obtained by deleting all chords of~$T$ crossing~$\chi$ and filling the resulting empty region~$R_\chi(T)$ by a centrally symmetric pseudotriangulation where all diagonals are incident to the point~$p$ and~$\bar p$ (note that the region~$R_\chi(T)$ needs not be connected). To define the pseudotriangulation replacing~$T$ inside~$R_\chi(T)$, imagine that all chords of~$T$ are rubber bands attached to their endpoints and pull the rubber bands crossed by $\chi$ along~$\chi$ towards~$p$ and~$\bar p$ to obtain~$\normalize_\chi(T)$. In other words, each diagonal~$\diag{x}{y}$ of~$T$ crossing~$\chi$ is replaced by:
\begin{itemize}
\item $\diag{x}{p}$ and~$\diag{p}{y}$ if it only crosses~$\diag{p}{q}$ and~$\bar p \notin \{x,y\}$ (and similarly exchanging~$p$~and~$\bar p$);
\item $\diag{x}{p}$, $\diagD{p}{R}$ and~$\diagD{\bar p}{L}$ if it only crosses~$\diag{p}{q}$, if~$x$ is on the right of the line from~$p$ to~$\bar p$, and if~$y = \bar p$ (and similarly exchanging~$p$~and~$\bar p$, or left and right);
\item $\diag{x}{p}$, $\diagD{p}{R}$, $\diagD{\bar p}{L}$, and~$\diag{\bar p}{y}$ if it crosses both~$\diag{p}{q}$ and~$\diag{\bar p}{\bar q}$, if~$\diag{p}{q}$ separates~$x$ from the origin while~$\diag{\bar p}{\bar q}$ separates~$y$ from the origin, and if the origin is on the left of the line from~$x$ to~$y$ (and similarly exchanging left and right).
\end{itemize}
Similarly, each left central chord~$\diagD{x}{L}$ (resp.~$\diagD{x}{D}$) of~$T$ crossing~$\chi$ is replaced by~$\diag{x}{p}$ and~$\diagD{p}{L}$ (resp.~$\diagD{p}{D}$) if it crosses~$\diag{p}{q}$ (and similarly for~$\diag{\bar p}{\bar q}$).
\item The normalization~$\normalize_\chi$ with respect to a centrally symmetric pair~$\chi$ of central chords transforms a centrally symmetric pseudotriangulation~$T$ into the centrally symmetric pseudotriangulation~$\normalize_\chi(T)$ obtained by deleting all chords of~$T$ crossing~$\chi$ and filling the resulting empty region~$R_\chi(T)$ by the left star~$\starL$ (resp.~right star~$\starR$) on this region if~$\chi$ is a pair of left (resp.~right) central chords.
\end{itemize}
\fref{fig:typeDNormalization} illustrates this normalization, where the three leftmost pictures correspond to the first situation, while the two rightmost correspond to the second situation. The same case analysis as in type~$A$, replacing the quadrangle~$Q$ by the pseudoquadrangle formed by glueing the two pseudotriangles of~$T$ incident to~$\delta$ (or equivalently the two pseudotriangles of~$T'$ incident to~$\delta'$) shows that~$\normalize_\chi$ indeed defines a normalization in both situations.

\begin{figure}[h]
	\centerline{\includegraphics[scale=.61]{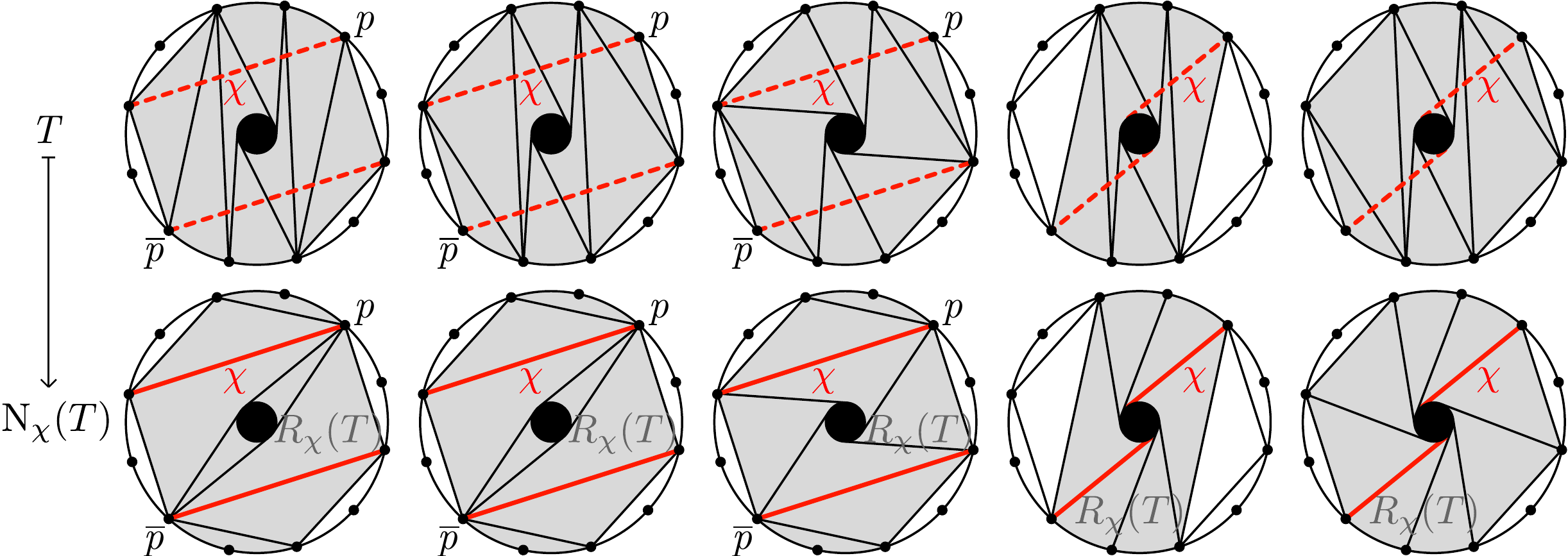}}
	\caption{Normalization map for type~$D$.}
	\label{fig:typeDNormalization}
\end{figure}


\subsection{Proof of Theorem~\ref{thm:nonLeavingFaceProperty} and Proposition~\ref{prop:enteringFaceProperty}}
\label{subsec:proofs}

Using the normalization introduced in the previous section, we are ready to prove the non-leaving-face property.

\begin{proof}[Proof of Proposition~\ref{prop:enteringFaceProperty}]
Consider two triangulations~$T$ and~$\widetilde T$ and a diagonal~$\chi$ in~$T \ssm \widetilde T$, such that the flip of~$\chi$ in~$T$ produces a diagonal~$\chi'$ that belongs to~$\widetilde T$. Let~$T = T_0, T_1, \dots, T_k = \widetilde T$ be an arbitrary geodesic between~$T$ and~$\widetilde T$. Consider the sequence~$T = T_0, \normalize_{\chi'}(T_0), \normalize_{\chi'}(T_1), \dots, \normalize_{\chi'}(T_k)$. By Property~(P1) above, the last triangulation in this sequence is~$\widetilde T$. By Property~(P3), the first two triangulations are connected by a flip. Since there exists at least one~$i$ such that~${\chi' \in T_{i+1} \ssm T_i}$, Property~(P3) also ensures that~$\normalize_{\chi'}(T_i) = \normalize_{\chi'}(T_{i+1})$. Finally, Property~(P2) asserts that any two consecutive triangulations in the remaining sequence either coincide or are adjacent. Erasing duplicated consecutive triangulations in this sequence gives a normalized path from~$T$ to~$\widetilde T$, which starts by the flip of~$\chi$, and is a geodesic since it is not longer than the geodesic~${T = T_0, T_1, \dots, T_k = \widetilde T}$.
\end{proof}


\begin{proof}[Proof of Theorem~\ref{thm:nonLeavingFaceProperty}]
Let~$T$ and~$\widetilde T$ be two triangulations. Consider a path~$\pi$ from~$T$ to~$\widetilde T$ which flips a common chord of~$T$ and~$\widetilde T$. Let~$T_1$ be the first triangulation along~$\pi$ which does not contain~$T \cap \widetilde T$, and let~$T_0$ the previous triangulation along~$\pi$. By Proposition~\ref{prop:enteringFaceProperty}, there exists a geodesic~$\pi'$ from~$T_1$ to~$\widetilde T$ starting with~$T_0$. Combining the subpath of~$\pi$ from~$T$ to~$T_0$ with the subpath of~$\pi'$ from~$T_0$ to~$\widetilde T$ produces a path from~$T$ to~$\widetilde T$ shorter than~$\pi$. It follows that no common chord of~$T$ and~$\widetilde T$ can be flipped along a geodesic between~$T$ and~$\widetilde T$.
\end{proof}


\section{Further examples related to the associahedron}
\label{sec:ctrexms}

To conclude, we survey five relevant generalizations of the flip graph on triangulations of a convex polygon. The diameter of these flip graphs is not precisely determined but we present some known asymptotic bounds. Interestingly, these five families contain specific examples that do not always satisfy the non-leaving-face property.


\subsection{Pseudotriangulation polytope}
\label{subsec:ctrexmPseudotriangulations}

A (pointed) \defn{pseudotriangulation} of a point set~$P$ in general position in the plane is a maximal set of edges of~$P$ which is crossing-free and pointed (any vertex is adjacent to an angle wider than~$\pi$). Flips between pseudotriangulations are defined as in item~(\ref{item:pseudoflip}) in Section~\ref{sec:Dmodel}.
We refer to~\cite{RoteSantosStreinu-survey} for a survey on pseudotriangulations and their properties. Using rigidity properties of pseudotriangulations and the expansive motion polyhedron, G.~Rote, F.~Santos and I.~Streinu showed in~\cite{RoteSantosStreinu-pseudotriangulationPolytope} that the flip graph on pseudotriangulations of~$P$ can be realized as the graph of the \defn{pseudotriangulation polytope}. This polytope is a realization of the type~$A$ associahedron when~$P$ is in convex position.

The diameter of the pseudotriangulation polytope of~$P$ is known to be bounded between~$|P|$ and~$|P| \cdot \log(|P|)$, see~\cite{Bereg} and~\cite{RoteSantosStreinu-survey}. These are the best current bounds to our knowledge.

O.~Aichholzer~\cite{Aichholzer} observed that not all pseudotriangulation polytopes satisfy the non-leaving-face property. His example is represented in \fref{fig:ctrexmPseudotriangulations}. The point set~$P$ is formed by a small downward triangle~$\triangledown$ together with a big upward triangle whose vertices are replaced by convex chains with~$m = 6$ points. \fref{fig:ctrexmPseudotriangulations} shows a path of~$3m+4 = 22$ flips from the top left pseudotriangulation~$T$ to the bottom left pseudotriangulation~$\widetilde T$, flipping common edges of~$T$ and~$\widetilde T$. In contrast, any path from~$T$ to~$\widetilde T$ which preserves their common edges has length at least~$4m-1 = 23$. Indeed, if the edges of~$T \cap \widetilde T$ are preserved, no flip produces an edge of~$\widetilde T$ before all but one edge incident to one vertex of the small downward triangle~$\triangledown$ have been flipped. We thus need at least~$m-1$ flips before a flip can create edges of~$\widetilde T \ssm T$, and then at least $3m$ flips to create them all.

\begin{figure}[h]
	\centerline{\includegraphics[width=1\textwidth]{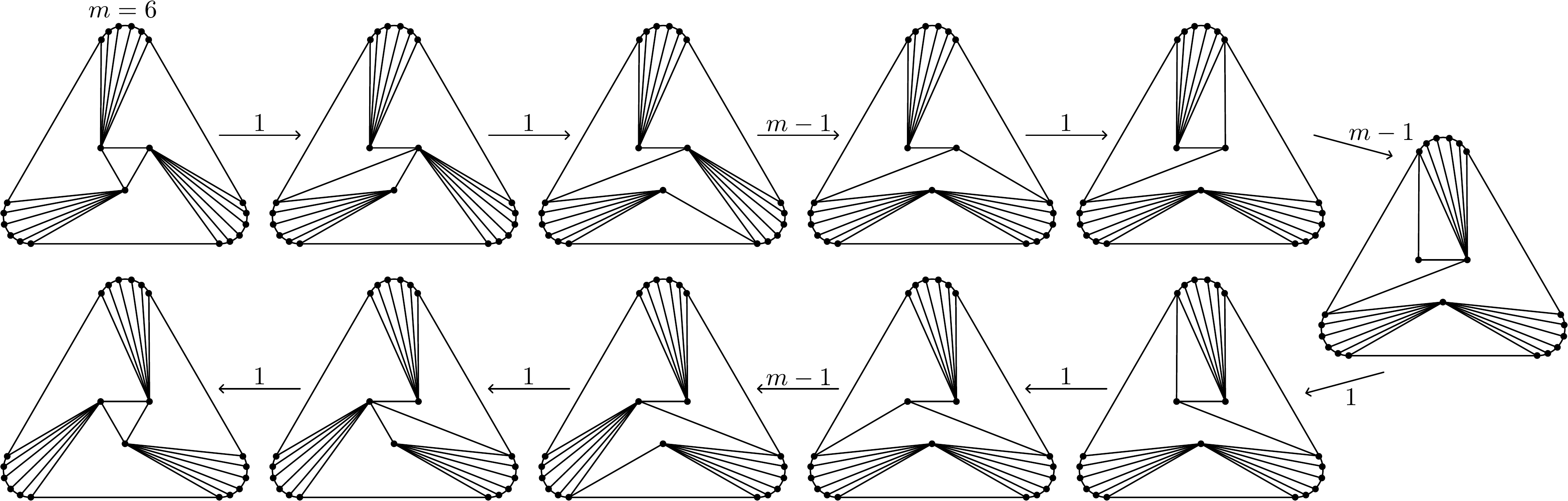}}
	\caption{A geodesic (of length~$3m+4 = 22$) between two pseudotriangulations that flips common edges between them.}
	\label{fig:ctrexmPseudotriangulations}
\end{figure}


\subsection{Multiassociahedron}
\label{subsec:ctrexmMultitriangulations}

A \defn{$k$-triangulation} of a convex $m$-gon is a maximal set of diagonals not containing a $(k+1)$-crossing, \ie a set of $k+1$ pairwise crossing diagonals. In a $k$-triangulation, every $k$-relevant diagonal (\ie with at least~$k$ vertices of the $m$-gon on each side) can be flipped to obtain a new $k$-triangulation. We refer to~\cite{PilaudSantos-multitriangulations} for a local description of this operation. The question whether the flip graph can be realized as the graph of a convex polytope remains open, except for very particular cases including the classical $n$-dimensional associahedron when~$k=1$ and~$m = n+3$. We can however study the diameter of the flip graph and the properties of its geodesics.

The best known bounds for the diameter~$\delta(m,k)$ of the flip graph on $k$-triangulations of the $m$-gon are given by
\[
\big( k+\frac{1}{2} \big) \cdot m - (k+1)^2 \; \le \; \delta(m,k) \; \le \; 2k \cdot m - 2k(4k+1).
\]
The upper bound, due to T.~Nakamigawa~\cite{Nakamigawa}, holds for all~$m \ge 4k^2(2k+1)$. Observe that it is tight when~$k=1$ by the result of~\cite{SleatorTarjanThurston, Pournin} on the associahedron. The lower bound, proved in~\cite[Lemma~2.39]{Pilaud-these}, holds for all~$m \ge 4k+2$. We refer to~\cite[Section~2.3.2]{Pilaud-these} for a summary of known properties on the diameter of the multiassociahedron.

We now show that the multiassociahedron does not satisfy the non-leaving-face property, \ie that not all $k$-triangulations along a geodesic between two $k$-triangulations always contain their common diagonals for large enough values of~$k$ and~$m$. We can argue using the universality property of multitriangulations~\cite[Proposition~5.6]{PilaudSantos-brickPolytope}. This property ensures in particular that the flip graph on pseudotriangulations of any planar point set in general position can be embedded as a subgraph induced by all $k$-triangulations containing a certain subset of diagonals in the flip graph on $k$-triangulations of an $m$-gon for large enough~$k$ and~$m$. The path of \fref{fig:ctrexmPseudotriangulations} thus results in a path in the flip graph on $k$-triangulations of the $m$-gon which contradicts the non-leaving-face property.

In fact, computer experiments with the software Sage~\cite{sage} provided us with a much smaller example illustrated on \fref{fig:ctrexmMultitriangulations}. This figure shows a path of~$4$ flips between the leftmost $2$-triangulation~$T$ and the rightmost $2$-triangulation~$\widetilde T$ of the octagon, such that a common diagonal is flipped along the sequence. In each $2$-triangulation, the blue dashed edge is inserted by the previous flip while the red bold edge is deleted by the next flip. Alternatively, it might help the reader to visualize these flips on the pseudoline arrangements represented below (see~\cite[Section~3.3]{PilaudPocchiola}). To see that the path represented in this figure is a geodesic, observe that~$T \ssm \widetilde T$ contains $3$ edges whose flip in~$T$ all produce an edge not in~$\widetilde T$.

\begin{figure}[h]
	\centerline{\includegraphics[width=1\textwidth]{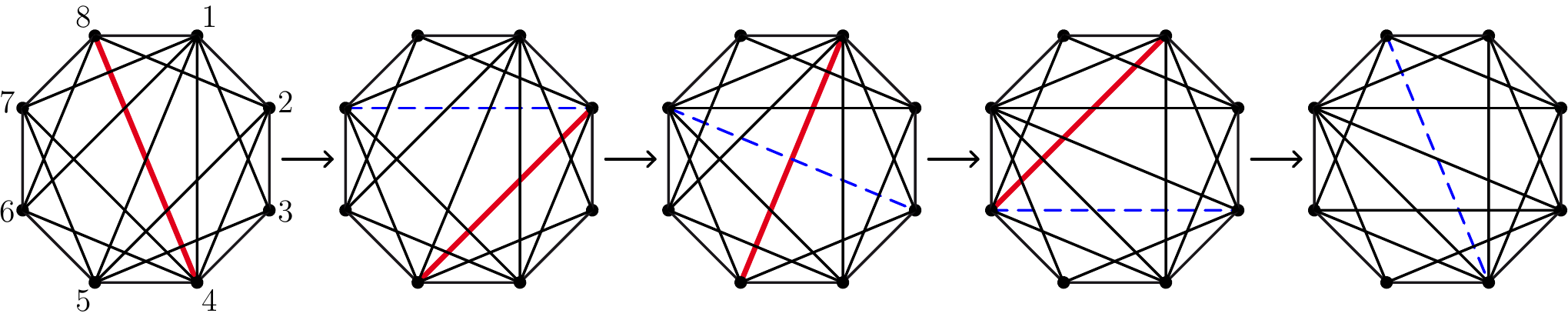}}\bigskip
	\centerline{
		\begin{overpic}[width=.19\textwidth]{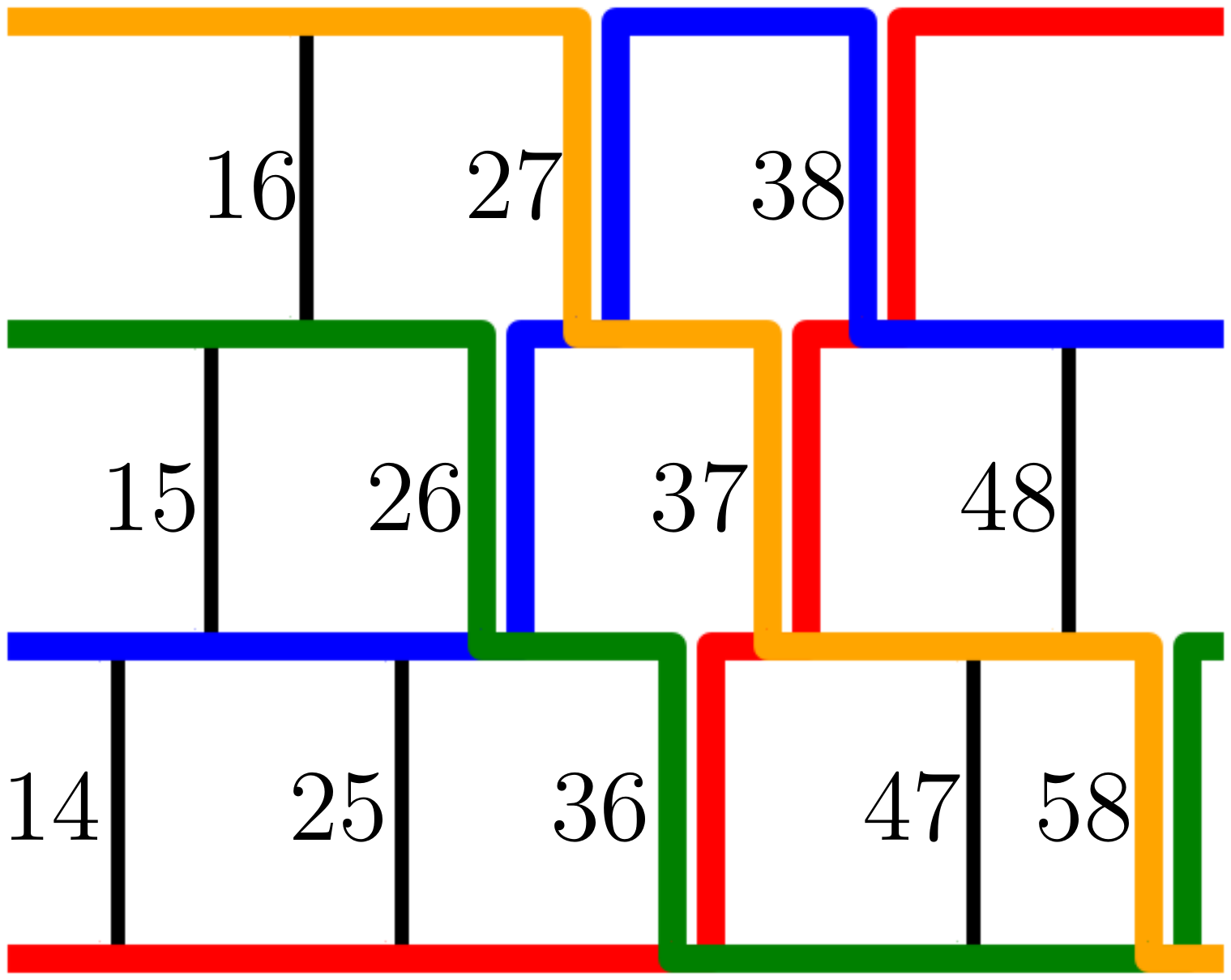}
			\put(0,12){\footnotesize $14$}
			\put(22,12){\footnotesize $25$}
			\put(42,12){\footnotesize $36$}
			\put(65,12){\footnotesize $47$}
			\put(79,12){\footnotesize $58$}
			\put(7,36){\footnotesize $15$}
			\put(27,36){\footnotesize $26$}
			\put(49,36){\footnotesize $37$}
			\put(73,36){\footnotesize $48$}
			\put(14,60){\footnotesize $16$}
			\put(34,60){\footnotesize $27$}
			\put(57,60){\footnotesize $38$}
		\end{overpic}
		\multido{\n=2+1}{4}{\;\;\includegraphics[width=.19\textwidth]{sortingNetwork\n}}
	}
	\caption{A geodesic between two $2$-triangulations of the octagon that flips a common edge between them.}
	\label{fig:ctrexmMultitriangulations}
\end{figure}


\subsection{Graph associahedron}
\label{subsec:ctrexmGraphAssociahedron}

Fix a finite connected graph~$G$. A \defn{nested set} on~$G$ is a set of \defn{tubes} (proper connected induced subgraphs) of~$G$ which are pairwise nested, or disjoint and non-adjacent in~$G$. The simplicial complex of all nested sets on~$G$ is called \defn{nested complex} and has been realized as the boundary complex of the \defn{$G$-associahedron}~\cite{CarrDevadoss}. Two maximal nested sets are connected by a \defn{flip} if one can be obtained from the other by replacing a tube with the unique other tube that produces a new nested set. Specific families of graph-associahedra provide relevant families of polytopes: path associahedra are classical type~$A$ associahedra, cycle associahedra are cyclohedra (or type~$B/C$ associahedra), star associahedra are stellohedra, and complete graph associahedra are type~$A$ permutahedra.

Graph properties of graph associahedra are studied by T.~Manneville and V.~Pilaud in~\cite{MannevillePilaud}. They show that the diameter~$\delta(G)$ of the $G$-associahedron is a non-decreasing function of the graph~$G$, meaning that~$\delta(G) \le \delta(G')$ if~$G$ is a subgraph of~$G'$. Combining this non-decreasing property with the diameter of the permutahedron on the one hand, and the analysis of the diameter~$\delta(T)$ for trees on the other hand (based on~\cite{Pournin}), they obtain that the diameter~$\delta(G)$ is bounded by
\[
\max(2n - 20, m) \; \le \; \delta(G) \; \le \; \binom{n}{2},
\]
where~$n$ and~$m$ denote the number of vertices and edges of~$G$.

Concerning the non-leaving-face property, they observed the following facts:
\begin{enumerate}[(i)]
\item Not all graph associahedra satisfy the non-leaving-face property. \fref{fig:ctrexmGraphAssociahedron} reproduces their example. It shows a path of length~$2n$ between two maximal nested sets~$N, \widetilde N$ on the star with~$n$ branches. In contrast, the minimal face containing~$N$ and~$\widetilde N$ is an $n$-dimensional permutahedron and the graph distance between~$N$ and~$\widetilde N$ in this face is~$\binom{n}{2}$. Therefore, the corresponding graph associahedron (stellohedron in this case) does not satisfy the non-leaving-face property when~$n \ge 5$.

\begin{figure}[h]
	\centerline{\includegraphics[width=1\textwidth]{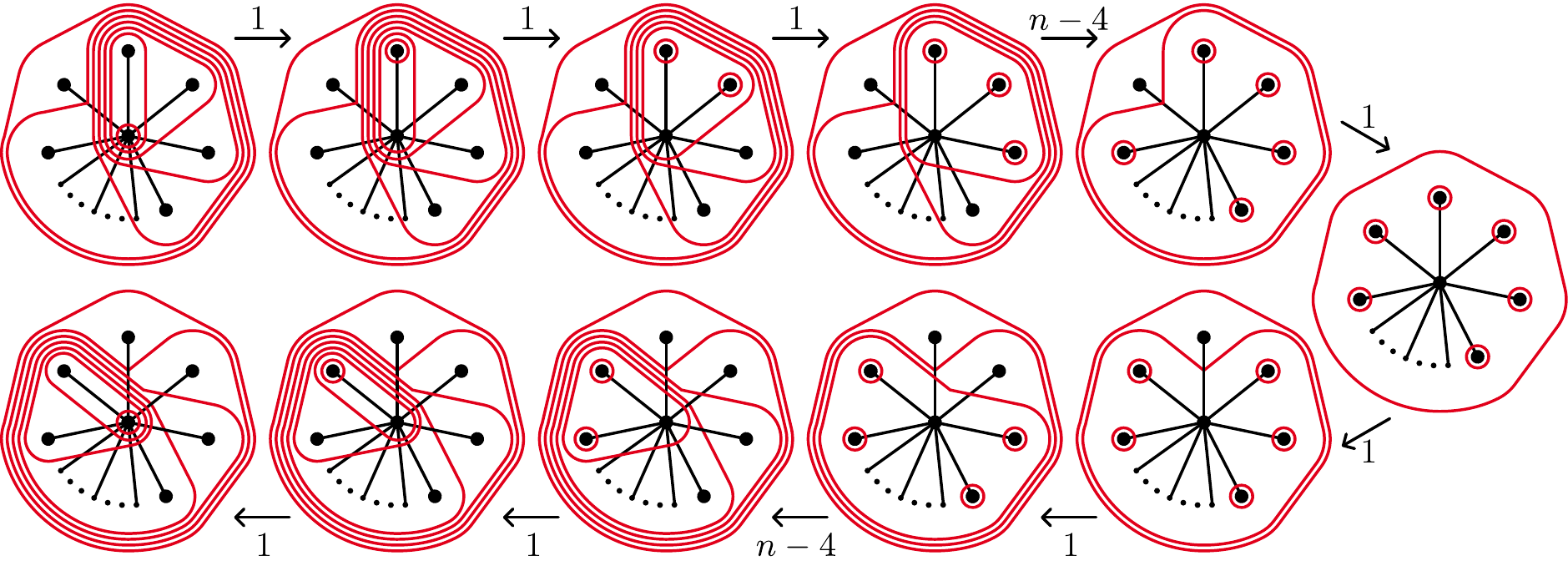}}
	\caption{A geodesic (of length~$2n$) between two maximal nested sets of the star that flips a common nested set.}
	\label{fig:ctrexmGraphAssociahedron}
\end{figure}

\item Any geodesic between two maximal nested sets~$N, \widetilde N$ on~$G$ remains in the face of the \mbox{$G$-asso}\-ciahedron corresponding to all common tubes of~$N$ and~$\widetilde N$ which are not contained in tubes of~$N \ssm \widetilde N$.
\end{enumerate}


\subsection{Secondary polytopes}
\label{subsec:ctrexmSecondaryPolytope}

The \defn{secondary polytope} of a point set~$P \subset \R^d$ is a polytope whose face lattice is isomorphic to the refinement poset of \defn{regular subdivisions} of~$P$, \ie polyhedral subdivisions of~$P$ which can be obtained as the vertical projection of the lower convex hull of the points of~$P$ lifted by an arbitrary height function. In particular, the graph of the secondary polytope of~$P$ has one vertex for each regular triangulation of~$P$ and one edge for each regular flip. See~\cite{LoeraRambauSantos} for a detailed presentation of this polytope. For example, the classical associahedron is the secondary polytope of a convex point set in the plane. The secondary polytope of a \mbox{$d$-dimen}\-sional configuration of~$n$ points has dimension~$n-d-1$, and it is known that its diameter cannot exceed
\[
\min\left( (d+2)\binom{n}{\left\lfloor \frac{d}{2} + 1 \right\rfloor}, \binom{n}{d+2} \right),
\]
see~\cite[Corollary~5.3.11]{LoeraRambauSantos}.

As observed by F.~Santos~\cite{Santos-Oberwolfach}, the secondary polytope of two dilated copies of the standard $(n-1)$-dimensional simplex is combinatorially equivalent to the graph associahedron of the star with $n$ leaves (stellohedron) discussed in Section~\ref{subsec:ctrexmGraphAssociahedron}. For example, \fref{fig:stellohedra} illustrates the correspondence between the secondary polytope of the ``mother of all examples''~\cite[Example~2.2.5]{LoeraRambauSantos} and the graph associahedron of the tripod. As illustrated in \fref{fig:ctrexmGraphAssociahedron}, the stellohedron of dimension~$n \ge 5$ does not satisfies the non-leaving-face property. Therefore, not all secondary polytopes satisfy the non-leaving-face property. 

\begin{figure}
	\centerline{\includegraphics[width=\textwidth]{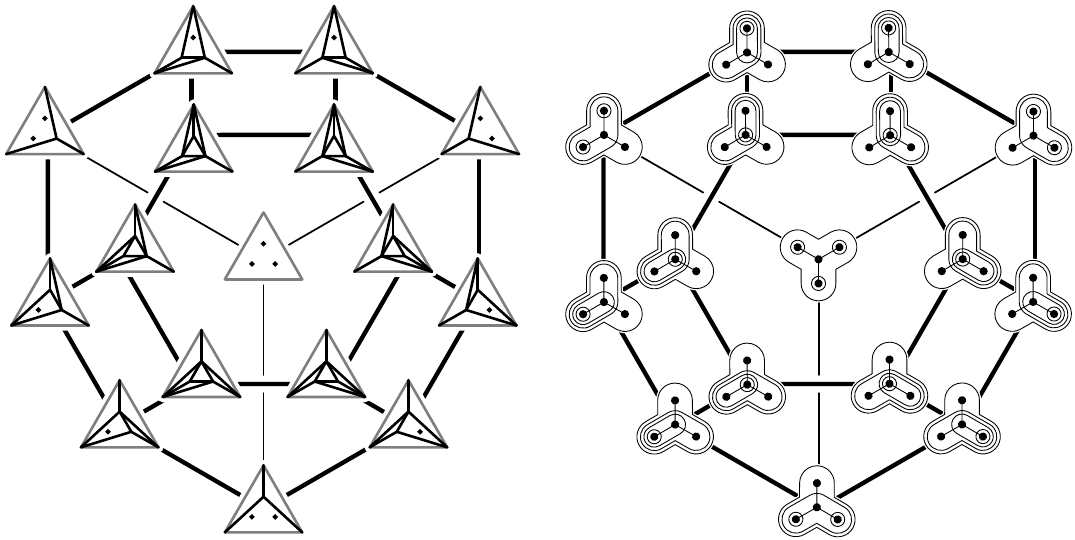}}
	\caption{The secondary polytope of two dilated copies of the standard simplex (left) is combinatorially equivalent to the stellohedron (right).}
	\label{fig:stellohedra}
\end{figure}


\subsection{Flip graph on all triangulations}
\label{subsec:ctrexmAllTriangulations}

Our last example is the flip graph on all triangulations (regular or not) of the point set $P$. This graph is not always connected~\cite{Santos}, but it is connected for point sets in the plane, in which case the diameter is at most~$4n$~\cite[Corollary~3.4.4]{LoeraRambauSantos}. This bound assumes that flips that insert or delete a vertex are allowed. Otherwise, the diameter can be become quadratic as illustrated by the ``double chain" example in~\cite[Example~3.4.5]{LoeraRambauSantos} reproduced in \fref{fig:ctrexmTriangulations}. This example also shows that geodesics between two triangulations in the flip graph delete and reinsert common edges between them.

\begin{figure}[h]
	\centerline{\includegraphics[width=.75\textwidth]{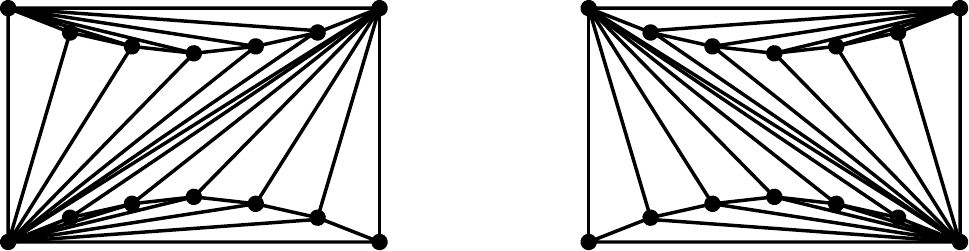}}
	\caption{Two triangulations of the double chain.}
	\label{fig:ctrexmTriangulations}
\end{figure}


\section*{Acknowledgments}

We are grateful to Oswin Aichholzer, to Thibault Manneville, and to Francisco Santos for letting us present the examples of Section~\ref{sec:ctrexms} and for interesting discussions related to them. We also thank Nantel Bergeron and Thibault Manneville for useful comments and suggestions on a preliminary draft. We thank Thomas Brustle for pointing out the work of Yannick Lebrun~\cite{Lebrun} to our attention. Finally, we are indebted to two anonymous referees for their careful reading and patient comments and suggestions which significantly improved the presentation and accuracy of the paper.


\bibliographystyle{alpha}
\bibliography{CeballosPilaud_diameterTypeDAssociahedron}

\newcommand{\etalchar}[1]{$^{#1}$}
\def\cprime{$'$}
\begin{thebibliography}{MHPS12}

\bibitem[AB14]{AdiprasitoBenedetti}
Karim~A. Adiprasito and Bruno Benedetti.
\newblock The {H}irsch conjecture holds for normal flag complexes.
\newblock {\em Math. Oper. Res.}, 39(4):1340--1348, 2014.

\bibitem[Aic10]{Aichholzer}
Oswin Aichholzer, 2010.
\newblock Personal communication, European Research Week on Geometric Graphs
  and Pseudotriangulations (Castelldefels).

\bibitem[Ber05]{Bereg}
Sergey Bereg.
\newblock Enumerating pseudo-triangulations in the plane.
\newblock {\em Comput.~Geom.}, 30(3):207--222, 2005.

\bibitem[CD06]{CarrDevadoss}
Michael~P. Carr and Satyan~L. Devadoss.
\newblock Coxeter complexes and graph-associahedra.
\newblock {\em Topology Appl.}, 153(12):2155--2168, 2006.

\bibitem[CFZ02]{ChapotonFominZelevinsky}
Fr{\'e}d{\'e}ric Chapoton, Sergey Fomin, and Andrei Zelevinsky.
\newblock Polytopal realizations of generalized associahedra.
\newblock {\em Canad. Math. Bull.}, 45(4):537--566, 2002.

\bibitem[CLS14]{CeballosLabbeStump}
Cesar Ceballos, Jean-Philippe Labb{\'e}, and Christian Stump.
\newblock Subword complexes, cluster complexes, and generalized
  multi-associahedra.
\newblock {\em J. Algebraic Combin.}, 39(1):17--51, 2014.

\bibitem[DRS10]{LoeraRambauSantos}
Jesus~A. {De Loera}, J\"org Rambau, and Francisco Santos.
\newblock {\em Triangulations: Structures for Algorithms and Applications},
  volume~25 of {\em Algorithms and {C}omputation in Mathematics}.
\newblock Springer Verlag, 2010.

\bibitem[FST08]{FominShapiroThurston}
Sergey Fomin, Michael Shapiro, and Dylan Thurston.
\newblock Cluster algebras and triangulated surfaces. {I}. {C}luster complexes.
\newblock {\em Acta Math.}, 201(1):83--146, 2008.

\bibitem[FZ02]{FominZelevinsky-ClusterAlgebrasI}
Sergey Fomin and Andrei Zelevinsky.
\newblock Cluster algebras. {I}. {F}oundations.
\newblock {\em J. Amer. Math. Soc.}, 15(2):497--529 (electronic), 2002.

\bibitem[FZ03a]{FominZelevinsky-ClusterAlgebrasII}
Sergey Fomin and Andrei Zelevinsky.
\newblock Cluster algebras. {II}. {F}inite type classification.
\newblock {\em Invent. Math.}, 154(1):63--121, 2003.

\bibitem[FZ03b]{FominZelevinsky-YSystems}
Sergey Fomin and Andrei Zelevinsky.
\newblock {$Y$}-systems and generalized associahedra.
\newblock {\em Ann. of Math. (2)}, 158(3):977--1018, 2003.

\bibitem[HLT11]{HohlwegLangeThomas}
Christophe Hohlweg, Carsten E. M.~C. Lange, and Hugh Thomas.
\newblock Permutahedra and generalized associahedra.
\newblock {\em Adv.~Math.}, 226(1):608--640, 2011.

\bibitem[Kle65]{Klee}
Victor Klee.
\newblock Paths on polyhedra. {I}.
\newblock {\em J. Soc. Indust. Appl. Math.}, 13:946--956, 1965.

\bibitem[Leb14]{Lebrun}
Yannick Lebrun.
\newblock Sur le diam\`etre du graphe d'€™\'echange de lâ€'alg\`ebre amass\'ee
  de type~${D}_n$.
\newblock {\em Cahiers Math\'ematiques de l'Universit\'e de Sherbrooke
  (CaMUS)}, 5:71--95, 2014.

\bibitem[MHPS12]{TamariFestschrift}
Folkert M{\"u}ller-Hoissen, Jean~Marcel Pallo, and Jim Stasheff, editors.
\newblock {\em Associahedra, {T}amari Lattices and Related Structures. Tamari
  Memorial Festschrift}, volume 299 of {\em Progress in Mathematics}.
\newblock Springer, New York, 2012.

\bibitem[MP14]{MannevillePilaud}
Thibault Manneville and Vincent Pilaud.
\newblock Graph properties of graph associahedra.
\newblock Preprint,
  \href{http://arxiv.org/abs/1409.8114}{\texttt{arXiv:1409.8114}}, 2014.

\bibitem[Nak00]{Nakamigawa}
Tomoki Nakamigawa.
\newblock A generalization of diagonal flips in a convex polygon.
\newblock {\em Theoret.~Comput.~Sci.}, 235(2):271--282, 2000.

\bibitem[Pil10]{Pilaud-these}
Vincent Pilaud.
\newblock {\em Multitriangulations, pseudotriangulations and some problems of
  realization of polytopes}.
\newblock PhD thesis, Universit\'e Paris 7 \& Universidad de Cantabria, 2010.
\newblock Available online
  \href{http://arxiv.org/abs/1009.1605}{\texttt{arXiv:1009.1605}}.

\bibitem[Pou14a]{Pournin-diameterTypeB}
Lionel Pournin.
\newblock The asymptotic diameter of cyclohedra.
\newblock Preprint,
  \href{http://arxiv.org/abs/1410.5259}{\texttt{arXiv:1410.5259}}, 2014.

\bibitem[Pou14b]{Pournin}
Lionel Pournin.
\newblock The diameter of associahedra.
\newblock {\em Adv. Math.}, 259:13--42, 2014.

\bibitem[PP12]{PilaudPocchiola}
Vincent Pilaud and Michel Pocchiola.
\newblock Multitriangulations, pseudotriangulations and primitive sorting
  networks.
\newblock {\em Discrete Comput. Geom.}, 48(1):142--191, 2012.

\bibitem[PS09]{PilaudSantos-multitriangulations}
Vincent Pilaud and Francisco Santos.
\newblock Multitriangulations as complexes of star polygons.
\newblock {\em Discrete~Comput.~Geom.}, 41(2):284--317, 2009.

\bibitem[PS12]{PilaudSantos-brickPolytope}
Vincent Pilaud and Francisco Santos.
\newblock The brick polytope of a sorting network.
\newblock {\em European~J.~Combin.}, 33(4):632--662, 2012.

\bibitem[PS15]{PilaudStump-brickPolytope}
Vincent Pilaud and Christian Stump.
\newblock Brick polytopes of spherical subword complexes and generalized
  associahedra.
\newblock {\em Adv. Math.}, 276:1--61, 2015.

\bibitem[Rea06]{Reading-CambrianLattices}
Nathan Reading.
\newblock Cambrian lattices.
\newblock {\em Adv.~Math.}, 205(2):313--353, 2006.

\bibitem[RS09]{ReadingSpeyer}
Nathan Reading and David~E. Speyer.
\newblock Cambrian fans.
\newblock {\em J.~Eur.~Math.~Soc.~(JEMS)}, 11(2):407--447, 2009.

\bibitem[RSS03]{RoteSantosStreinu-pseudotriangulationPolytope}
G{\"u}nter Rote, Francisco Santos, and Ileana Streinu.
\newblock Expansive motions and the polytope of pointed pseudo-triangulations.
\newblock In {\em Discrete and computational geometry}, volume~25 of {\em
  Algorithms Combin.}, pages 699--736. Springer, Berlin, 2003.

\bibitem[RSS08]{RoteSantosStreinu-survey}
G{\"u}nter Rote, Francisco Santos, and Ileana Streinu.
\newblock Pseudo-triangulations~---~a survey.
\newblock In {\em Surveys on discrete and computational geometry}, volume 453
  of {\em Contemp.~Math.}, pages 343--410. Amer. Math. Soc., Providence, RI,
  2008.

\bibitem[S{\etalchar{+}}12]{sage}
William~A. Stein et~al.
\newblock {\em {S}age {M}athematics {S}oftware ({V}ersion 4.8)}.
\newblock The Sage Development Team, 2012.
\newblock \href{http://www.sagemath.org}{\texttt{http://www.sagemath.org}}.

\bibitem[San00]{Santos}
Francisco Santos.
\newblock A point set whose space of triangulations is disconnected.
\newblock {\em J. Amer. Math. Soc.}, 13(3):611--637 (electronic), 2000.

\bibitem[San13]{Santos-surveyHirsch}
Francisco Santos.
\newblock Recent progress on the combinatorial diameter of polytopes and
  simplicial complexes.
\newblock {\em TOP}, 21(3):426--460, 2013.

\bibitem[San15]{Santos-Oberwolfach}
Francisco Santos, 2015.
\newblock Personal communication, Oberwolfach workshop on Geometric and
  Algebraic Combinatorics.

\bibitem[Sta63]{Stasheff}
Jim Stasheff.
\newblock Homotopy associativity of {H}-spaces {I}, {II}.
\newblock {\em Trans. Amer. Math. Soc.}, 108(2):293--312, 1963.

\bibitem[Sta97]{Stasheff-operads}
Jim Stasheff.
\newblock From operads to ``physically'' inspired theories.
\newblock In {\em Operads: Proceedings of Renaissance Conferences (Hartfort,
  CT/Luminy, 1995)}, volume 202 of {\em Contemporary Mathematics}, pages
  53--81, Cambridge, MA, 1997. American Mathematical Society.

\bibitem[STT88]{SleatorTarjanThurston}
Daniel~D. Sleator, Robert~E. Tarjan, and William~P. Thurston.
\newblock Rotation distance, triangulations, and hyperbolic geometry.
\newblock {\em J. Amer. Math. Soc.}, 1(3):647--681, 1988.

\bibitem[Wil15]{Williams}
Nathan Williams.
\newblock $w$-{A}ssociahedra are {I}n-{Y}our-{F}ace.
\newblock Preprint,
  \href{http://arxiv.org/abs/1502.01405}{\texttt{arXiv:1502.01405}}, 2015.

\bibitem[Zie95]{Ziegler1}
G{\"u}nter~M. Ziegler.
\newblock {\em Lectures on polytopes}, volume 152 of {\em Graduate Texts in
  Mathematics}.
\newblock Springer-Verlag, New York, 1995.

\end{thebibliography}
\label{sec:biblio}


\begin{landscape}
\pagestyle{empty}


\begin{figure}[p]
	\vspace*{-2.5cm}
	\centerline{\includegraphics[width=1.45\textwidth]{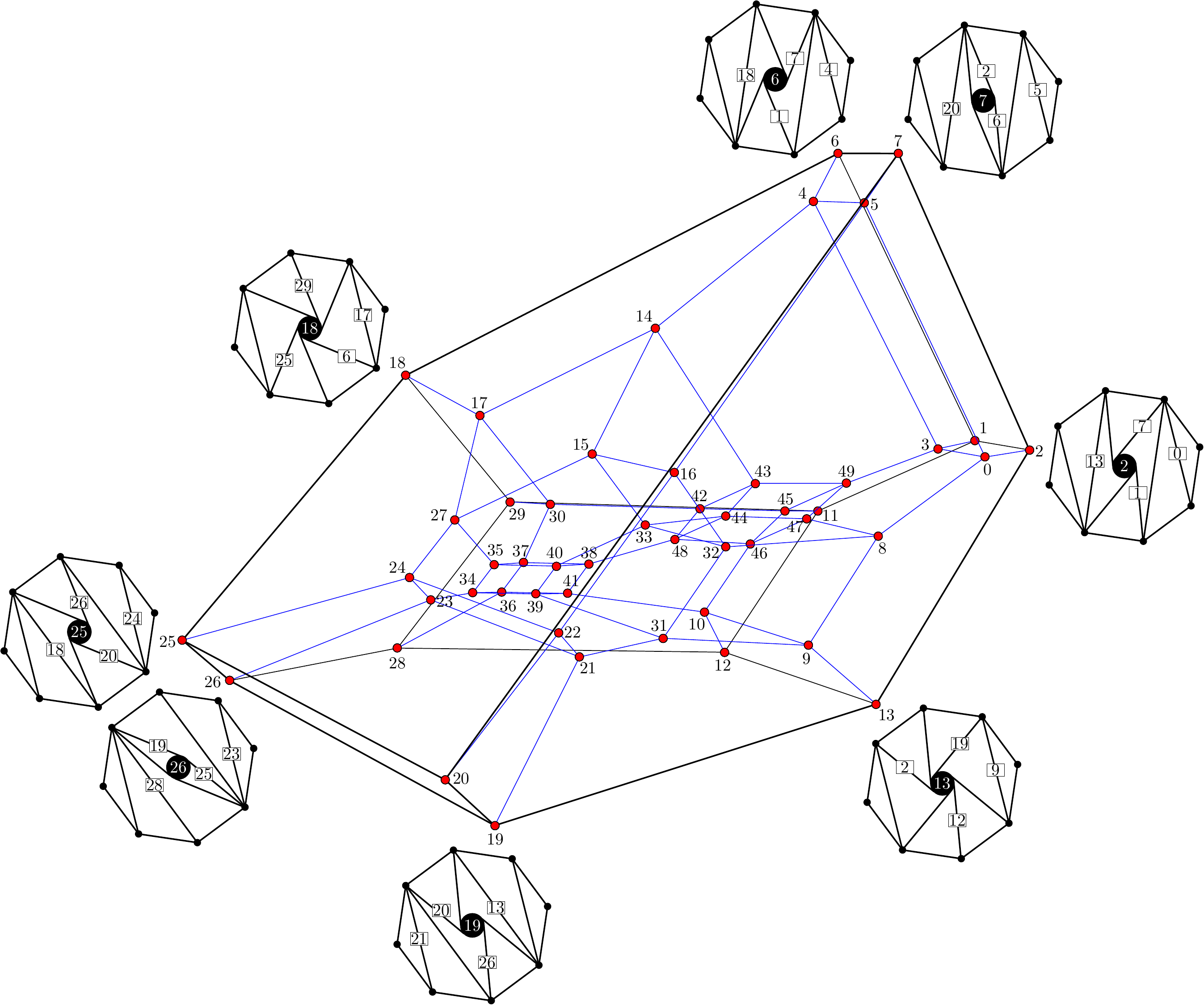}}
	\caption{The type~$D_4$ flip graph. We have represented some of the corresponding centrally symmetric pseudotriangulations of~$\configD_4$ on this picture, while the others can be found on \fref{fig:42pseudotriangulationsHorizontal}. In each pseudotriangulation, the number at the center of the disk is its label in the flip graph, and each pair of chords is labeled with the pseudotriangulation obtained when flipping it. The underlying graph used for the representation is a Schlegel diagram of the type~$D_4$ associahedron~\cite{ChapotonFominZelevinsky, HohlwegLangeThomas, PilaudStump-brickPolytope}.}
	\label{fig:typeD4associahedron}
\end{figure}

\begin{figure}[p]
	\vspace*{-2.5cm}
	\centerline{\includegraphics[width=1.45\textwidth]{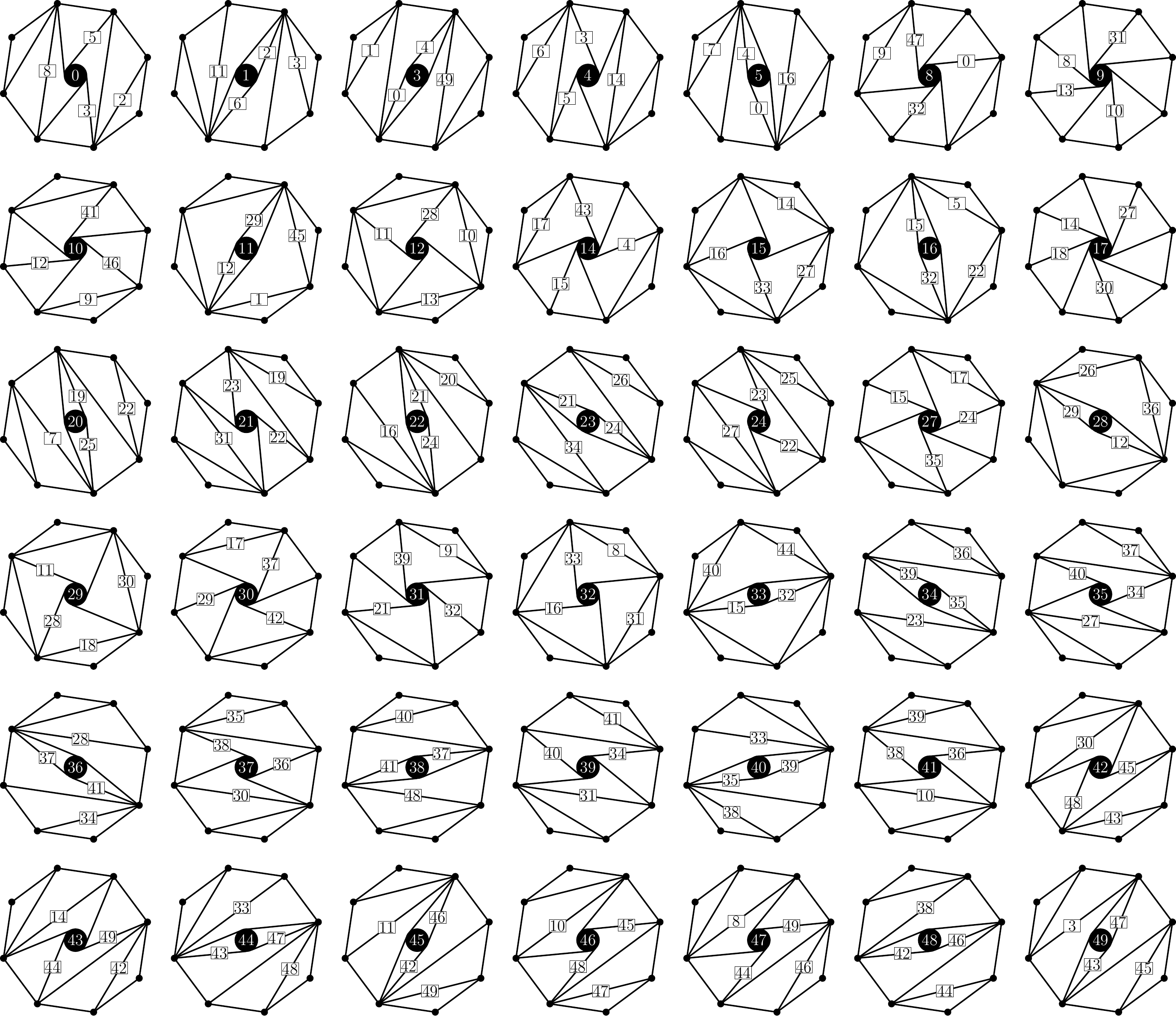}}
	\caption{The remaining~$42$ centrally symmetric pseudotriangulations of the configuration~$\configD_4$. See \fref{fig:typeD4associahedron} for the other~$8$ centrally symmetric pseudotriangulations, the flip graph and the explanation of the labeling conventions.}
	\label{fig:42pseudotriangulationsHorizontal}
\end{figure}

\end{landscape}

\end{document}